\documentclass[aap,MSNbibl,citesort,dvips]{arximspdf}
\usepackage{subenv}

\doi{10.1214/10-AAP708}
\volume{21}
\issue{2}
\pubyear{2011}
\firstpage{669}
\lastpage{698}

\makeatletter

\newcommand{\rmE}{{e}}

\newcommand{\C}{\mathbb{C}}
\newcommand{\CB}{\mathcal{B}}
\newcommand{\CC}{\mathcal{C}}
\newcommand{\CD}{\mathcal{D}}
\newcommand{\CH}{\mathcal{H}}
\newcommand{\CL}{\mathcal{L}}
\newcommand{\CN}{\mathcal{N}}
\newcommand{\CO}{\mathcal{O}}
\newcommand{\Cov}{\operatorname{Cov}}
\newcommand{\CP}{\mathcal{P}}
\newcommand{\CX}{\mathcal{X}}
\newcommand{\E}{\mathbb{E}}
\newcommand{\Lip}{\operatorname{Lip}}
\newcommand{\N}{\mathbb{N}}
\newcommand{\quark}{\cdot}
\newcommand{\R}{\mathbb{R}}
\newcommand{\range}{\operatorname{Range}}
\newcommand{\rC}{\mathrm{C}}
\newcommand{\rH}{\mathrm{H}}
\newcommand{\rL}{\mathrm{L}}
\newcommand{\tr}{\operatorname{tr}}
\newcommand{\vspan}{\operatorname{span}}

\newtheorem{theorem}{Theorem}
\newtheorem{lemma}[theorem]{Lemma}
\newtheorem{corollary}[theorem]{Corollary}
\newtheorem{proposition}[theorem]{Proposition}

\newproclaim{definition}[theorem]{Definition}
\newproclaim{remark}[theorem]{Remark}

\makeatother

\begin{document}
\begin{frontmatter}

\title{Sampling conditioned hypoelliptic diffusions\thanksref{T1}}
\runtitle{Sampling conditioned hypoelliptic diffusions}

\thankstext{T1}{Supported by EPSRC Grant EP/E002269/1.}

\begin{aug}
\author[A]{\fnms{Martin} \snm{Hairer}\ead
[label=e1]{M.Hairer@warwick.ac.uk}\ead[label=e4]{mhairer@cims.nyu.edu}},
\author[A]{\fnms{Andrew M.} \snm{Stuart}\thanksref{t2}\ead
[label=e2]{A.M.Stuart@warwick.ac.uk}} and
\author[B]{\fnms{Jochen} \snm{Voss}\corref{}\ead
[label=e3]{J.Voss@leeds.ac.uk}}
\runauthor{M. Hairer, A. M. Stuart and J. Voss}
\affiliation{University of Warwick, University of Warwick and
University of Leeds}
\address[A]{M. Hairer\\
A. Stuart\\
Mathematics Institute\\
University of Warwick\\
Coventry CV4 7AL\\
United Kingdom\\
\printead{e1}\\
\phantom{E-mail: }\printead*{e4}\\
\phantom{E-mail: }\printead*{e2}}
\address[B]{J. Voss\\
Department of Statistics\\
University of Leeds\\
Leeds LS2 9JT\\
United Kingdom\\
\printead{e3}}
\end{aug}

\thankstext{t2}{Supported by ERC and EPSRC.}
%part of this article was completed, for its hospitality.}

% HISTORY:
\received{\smonth{8} \syear{2009}}
\revised{\smonth{4} \syear{2010}}

% ABSTRACT
%
\begin{abstract}
A series of recent articles introduced a method to construct
stochastic partial differential equations (SPDEs) which are
invariant with respect to the distribution of a given conditioned
diffusion. These works are restricted to the case of elliptic
diffusions where the drift has a gradient structure and the
resulting SPDE is of second-order parabolic type.

The present article extends this methodology to allow the
construction of SPDEs which are invariant with respect to the
distribution of a class of hypoelliptic diffusion processes, subject
to a bridge conditioning, leading to SPDEs which are of fourth-order
parabolic type. This allows the treatment of more realistic
physical models, for example, one can use the resulting SPDE to study
transitions between meta-stable states in mechanical systems with
friction and noise. In this situation the restriction of the drift
being a gradient can also be lifted.
\end{abstract}

% KEYWORDS
%

\end{frontmatter}

%s1 ###
\section{Introduction}

In previous works (see, e.g.,
\cite{StuartVossWiberg04,HairerStuartVossWiberg05,HairerStuartVoss07} or
\cite{HairerStuartVoss08} for a review) we described an SPDE-based
method to sample paths from SDEs of the form
%
%e1 ###
%
\begin{equation}\label{E:sde1}
\dot x(t) = f(x(t)) + \dot w(t)\qquad
\forall t\in[0,T],
\end{equation}
where $\dot w$ is white noise, conditioned on several different types
of events. The method works by introducing an ``algorithmic time''
$\tau$ and constructing a second-order SPDE of the form
%
%e2 ###
%
\begin{eqnarray}\label{E:spde1}
\partial_\tau x(\tau, t)
= \partial_t^2 x(\tau,t )
+ \CN(x(\tau,t))
+ \sqrt{2} \,\partial_\tau w(\tau, t)\nonumber\\[-8pt]\\[-8pt]
&&\eqntext{\forall(\tau,t)\in\R_+\times[0,T],}
\end{eqnarray}
which has $t$ as its space variable. Here $\partial_\tau w(\tau, t)$ is
space--time white noise. The nonlinearity $\CN$ and the boundary
conditions of the differential operator $\partial_t^2$ are constructed such
that, in stationarity, the distribution of the random function $t
\mapsto x(\tau, t)$ coincides with the required conditioned
distribution. See also \cite{Reznikoff-VandenEijnden05}.
It transpires that the distribution of (\ref{E:sde1})
under the bridge conditions $x(0) = x(T) = 0$ corresponds the choice
\[
\CN_j(x) = - f_i(x) \,\partial_j f_i(x) - \tfrac12 \,\partial_{ij}^2 f_i(x)
\]
(written using Einstein's summation convention) and use of Dirichlet
boundary conditions for $\partial_t^2$.

Assuming ergodicity of the sampling SPDE, one can now solve the
sampling problem by simulating a solution to (\ref{E:spde1}) up to a
large time $\tau$ and then taking $t \mapsto x(\tau, t)$ as an
approximation to a path from the conditioned SDE. The resulting
sampling method has many applications, some of which are described in
\cite{ApteHairerStuartVoss07}. The
biggest restrictions of this method are that the derivation requires
the drift $f$ to have some gradient structure and the diffusion matrix
[chosen to be the identity matrix in (\ref{E:sde1}) above] to be
invertible.

In this article we consider the different problem of sampling
conditioned paths of the second-order SDE
%
%e3 ###
%
\begin{equation}\label{E:sde2}
m \ddot x(t)
= f(x(t))
- \dot x(t)
+ \dot w(t)\qquad
\forall t\in[0,T],
\end{equation}
conditioned on $x(0) = x_-$ and $x(T) = x_+$.
Equation (\ref{E:sde2}) could, for example, describe the time
evolution of a noisy mechanical system with inertia and friction.
Rewriting this second-order SDE as a system of first-order SDEs for
$x$ and $\dot x$ leads to a drift which is in general not a gradient
(even in the case when $f$ itself is one) and, since the noise only
acts on $\dot x$, one obtains a singular diffusion matrix. Thus, this
problem is outside the scope of the previous results. However,
it has enough structure so that it still can
be treated within a similar framework. Indeed, we
derive in this article a fourth-order SPDE of the form
%
%e4 ###
%
\begin{eqnarray}\label{E:spde2}
\partial_\tau x(\tau,t)
= (\partial_t^2 - m^2\,\partial_t^4) x(\tau,t)
+ \CN(x)(\tau,t)
+ \sqrt{2} \,\partial_\tau w(\tau, t)\nonumber\\[-8pt]\\[-8pt]
&&\eqntext{\forall(\tau,t)\in\R_+\times[0,T],}
\end{eqnarray}
where, again, the boundary conditions and the drift term $\CN$ are
chosen in such a way that the conditioned distribution
of (\ref{E:sde2}) is stationary for (\ref{E:spde2}).

One surprising fact about this result is that it does \textit{not}
require $f$
to be a gradient. In our earlier works, even the appropriate
definition of solutions for the (formal) second-order SPDE derived to
sample conditioned paths of (\ref{E:sde1}) in the nongradient case is
not clear (see \cite{HairerStuartVoss07}, Section 9 or
\cite{ApteHairerStuartVoss07}, Section 9.2); the analysis for elliptic
equations is thus restricted to the gradient case. In contrast, the
greater regularity of solutions to SPDE (\ref{E:spde2}) here, sampling
conditioned paths of (\ref{E:sde2}), allows us to obtain existence
results for the fourth-order SPDEs arising without any gradient
requirements on $f$.

In the special case where $f$ is a gradient and $f(x_-) = f(x_+) = 0$,
the components of the nonlinearity $\CN$ can be written as
\[
\CN_j(x) =
- f_i(x) \,\partial_j f_i(x)
- m \,\partial_t x_i \,\partial_t x_k \,\partial_{jk}^2 f_i(x)
+ m \,\partial_t \bigl( \partial_t x_i(\partial_i f_j(x) +
\partial_j f_i(x)) \bigr)
\]
(using Einstein's summation convention again). It is tempting to try
to derive (\ref{E:spde1}) by taking the limit $m\downarrow0$
in (\ref{E:spde2}), in particular since the first terms of the
corresponding nonlinearities coincide. It transpires that taking this
limit is not entirely trivial: one needs to argue that on one hand,
$m\,\partial_t x_i \,\partial_t x_k \to\frac12 \delta_{ik}$ as
$m\downarrow0$, but
that the term $m \,\partial_t [ \partial_t x_i (\partial_i f_j +
\partial_j f_i)]$
becomes negligible in the limit. Nevertheless, this argument can be
made exact; see \cite{Hairer10} for a rigorous derivation of the
required limiting procedure.

One novelty of this article compared to earlier work like
\cite{Zabczyk88,HairerStuartVoss07} is that there is no natural Banach
space (like the space of continuous functions) on which the
nonlinearity is well defined and on which the linearized equation
generates a contraction semigroup. The reason for this is that the
linear operator of the equations studied in this article is a
fourth-order differential operator. Another technical difficulty
stems from the fact that the nonlinearity $\CN$ has very weak
dissipativity and regularity properties.

While preparing this text, we performed some numerical simulations on
the fourth-order SPDE presented here. Our aim was to study whether
the SPDE could be used as the basis of an infinite-dimensional MCMC
method. Different from the situation in earlier articles, these
simulations proved prohibitively slow and the resulting method does
not seem like a useful approach to sampling. This is mainly due to
the fact that the convergence time to equilibrium seems to grow like $T^4$
and thus can get very big for nontrivial problems. In the gradient
case, since the system converges to the second-order SPDE as $m\to0$,
one could expect improved convergence rates for small values of $m$.
However, the theory developed in \cite{Hairer10} suggests that the
relevant lengthscale for the small-$m$ problem is $m$, suggesting that
one would need numerical simulations that resolve significantly
smaller scales than that in order to obtain reliable results. Again,
this would lead to inefficient numerical methods even in the case of
small $m$. Consequently, we do not include our simulation results in
this article.

For a number of articles considering fourth-order (S)PDEs, see, for
example, \mbox{\cite{BloMaPaWa01,DaPraDe96}} and \cite{LauCueMa96}.
Alternative methods to construct solutions of SPDEs and to identify
their stationary distributions are based on the theory of Dirichlet
forms (see, e.g., \cite{RoMa92}).

The text is structured as follows: in Section \ref{S:problem} we give
a detailed description of the sampling problem under consideration and
formulate the main result in Theorem \ref{T:answer}. The proof of
this result is given in Sections \ref{S:operator}, \ref{S:linear}
and \ref{S:nonlinear}.

\subsection*{Notation}

Throughout the article we will use the notation as introduced above: by $s,
t\in[0,T]$ we denote ``physical time,'' that is, the time
variable in equations like (\ref{E:sde1}) and (\ref{E:sde2}) which
define the target distributions. By $\sigma, \tau\geq0$ we denote
``algorithmic time,'' that is, the time variable in sampling
equations like (\ref{E:spde1}) and (\ref{E:spde2}). Thus, in the
sampling SPDEs, $\tau$ takes the role of time and $t$ takes the role
of space.

%%%%%%%%%%%%%%%%%%%%%%%%%%%%%%%%%%%%%%%%%%%%%%%%%%%%%%%%%%%%%%%%%%%%%%

%s2 ###
\section{The sampling problem}
\label{S:problem}

In this section we give the full statement of the sampling problem we
want to solve; the main result is contained in Theorem \ref{T:answer}.

First consider the following unconditioned second-order SDE:
%
%e5 ###
%
\begin{eqnarray}\label{E:sde2-careful}
m \ddot x(t)
&=& f(x(t)) - \dot x(t) + \dot w(t) \qquad\forall t\in
[0,T], \nonumber\\[-8pt]\\[-8pt]
x(0) &=& x_0,\qquad \dot x(0) = v_0,
\nonumber
\end{eqnarray}
where the solution $x$ takes values in $\R^d$, $m>0$ is a constant,
$f\dvtx\R^d\to\R^d$ is a given function and $w$ is a standard
Brownian motion on $\R^d$. The initial conditions $x_0$ and $v_0$ are
either deterministic or random variables independent of $w$. The
solution to this SDE can be interpreted as the time evolution of the
state of a mechanical system with friction under the influence of
noise. In this case $m$ would be the \textit{mass} and $f$ would be
an \textit{external force field}. Models like this are, for example,
widely used in molecular dynamics since, for conservative forces $f$,
they describe Hamiltonian systems in contact with a heat bath. In
this context, equation (\ref{E:sde2-careful}) is called the
\textit{Langevin equation}. The limiting case $m=0$ corresponds to
the \textit{Brownian dynamics} (\ref{E:sde1}).
\begin{remark}\label{R:rescaling}
Arbitrary constants in front of the $\dot x$ and $\dot w$ terms can
be introduced using a scaling argument: let $\beta, \gamma>0$ and
define the process $y$ by $y(t) = x(t / \gamma) / \sqrt{\beta/2}$.
Then $y$ solves the SDE
\[
\tilde m \ddot y(t)
= \tilde f(y(t))
- \gamma \dot y(t)
+ \sqrt{\frac{2\gamma}{\beta}} \dot w(t)\qquad
\forall t\in[0, \tilde T],
\]
where $\tilde m = \gamma^2 m$, $\tilde f(x) = f(\sqrt{\beta/2}
x)/\sqrt{\beta/2}$ and $\tilde T = \gamma T$. Thus, by
rescaling $T$, $m$ and $F$ we can assume $\beta= 2$ and $\gamma=
1$ without loss of generality.
\end{remark}

For our analysis we rewrite the second-order
SDE (\ref{E:sde2-careful}) as a system of first-order SDEs in the
variables $x$ and $\dot x$. We get
%
%e6 ###
%
\begin{eqnarray}\label{E:target}
dx(t) &=& \dot x(t) \,dt,\qquad
x(0) = x_0, \nonumber\\[-8pt]\\[-8pt]
m d\dot x(t) &=& f(x(t)) \,dt - \dot x(t) \,dt
+ dw(t),\qquad
\dot x(0) = v_0.
\nonumber
\end{eqnarray}

In the Hamiltonian case $f(x) = - \nabla V(x)$ and provided that the
potential $V$ is sufficiently regular, it can be checked that the
Boltzmann--Gibbs distribution
\[
\exp\bigl(-2 \bigl(V(x) + \tfrac12 m \dot x^2 \bigr)\bigr) \,d\dot
x \,dx
\]
is invariant for (\ref{E:target}). If $V$ is sufficiently coercive,
this distribution can be normalized to a probability
distribution. Note that in equilibrium, the position $x$ and the
velocity $\dot x$ are independent. Thus, in stationarity, the
velocity satisfies $\dot x(t) \sim\CN(0, 1/2m)$ for all
$t\in[0,T]$. We will, even for the nongradient case, use this
distribution for the initial condition for $\dot x$.
\begin{definition}\label{D:target}
For $T>0$, $x_-,x_+\in\R^d$ and\vspace*{1pt} $f\dvtx\R^d\to\R^d$, define
$Q_f^{0,x_-}$ to be the distribution of the process $x$ given by
(\ref{E:target}) where $x_0=x_-$ and $v_0 \sim\CN(0,1/2m
)$, independent of $w$. Define the \textit{target
distribution} $Q_f^{0,x_-;T,x_+}$ to be the distribution of $x$
under $Q_f^{0,x_-}$, conditioned on $x(T) = x_+$.
\end{definition}

The sampling problem considered in the rest of this article is to find
a stochastic process with values in $\rL^2([0,T], \R^d)$
which has the target distribution $Q_f^{0,x_-;T,x_+}$ as its
stationary distribution. Note that $Q_f^{0,x_-;T,x_+}$ is just the
distribution of $x$ and not of the pair $(x,\dot x)$ and thus is a
probability measure on $\rL^2([0,T], \R^d)$. Considering
this distribution is sufficient since for solutions
of (\ref{E:target}) the initial condition $x(0)=x_-$ allows to find a
bijection between the paths $x$ and $\dot x$. If $f$ is a gradient,
the distribution $Q_f^{0,x_-;T,x_+}$ coincides with the distribution
of the process in stationarity, conditioned on $x(0)=x_-$
and $x(T)=x_+$.
\begin{definition}\label{def:L-and-mean}
Let $L$ denote the formal differential operator
\[
L = - m^2 \,\partial_t^4 + \partial_t^2
\]
and define $\CL$ to be this differential operator on the space
$\rL^2([0,T], \R^d)$ equipped with the domain
%
%e7 ###
%
\begin{eqnarray}\label{E:bch}
\CD(\CL) &=& \{ x\in\rH^4 \mid
x(0)=x(T)=0,\nonumber\\[-8pt]\\[-8pt]
&&\hspace*{4.1pt}m \,\partial_t^2 x(0) = \partial_t x(0),
m \,\partial_t^2 x(T) = -\partial_t x(T) \},\nonumber
\end{eqnarray}
where $\rH^4 = \rH^4([0,T], \R^d)$ is the Sobolev space of
functions with square integrable generalized derivatives up to the
fourth order. Furthermore, let $\bar x\dvtx[0,T]\to\R^d$ be the
solution of the boundary value problem $L \bar x = 0$ with boundary
conditions
%
%e8 ###
%
\begin{eqnarray}\label{E:xbar-bc}
\bar x(0)&=&x_-,\qquad
\bar x(T)=x_+,\nonumber\\[-8pt]\\[-8pt]
m \,\partial_t^2 \bar x(0 ) &=& \partial_t \bar x(0),\qquad
m \,\partial_t^2 \bar x(T) = - \partial_t \bar x(T).\nonumber
\end{eqnarray}
\end{definition}

We will see in Lemma \ref{L:stat-dist-linear} that the operator $\CL$
given by this definition is self-adjoint and negative definite.
\begin{theorem}\label{T:answer}
Consider the $\rL^2([0,T], \R^d)$-valued equation
%
%e9 ###
%
\begin{equation}\label{E:SPDE}
dx(\tau) = \CL\bigl(x(\tau)-\bar x\bigr) \,d\tau
+ \CN(x(\tau)) \,d\tau
+ \sqrt{2} \,dw(\tau),\qquad
x(0) = x_0.
\end{equation}
Here $\CL$ and $\bar x$ are given in
Definition \ref{def:L-and-mean}, $w$ is a cylindrical Wiener
process, $x_0\in\rL^2([0,T],\R^d)$ and
%
%e10 ###
%
\begin{eqnarray}\label{E:drift}
\CN_k(x) &=&
- f_i(x) \,\partial_k f_i(x)
+ m \,\partial_t x_i \,\partial_t x_j \,\partial_{ij}^2
f_k(x) \nonumber\\
&&{} - \partial_t x_i \bigl(\partial_i f_k(x) - \partial_k
f_i(x)\bigr)
+ m \,\partial_t^2 x_i \bigl(\partial_i f_k(x) + \partial_k
f_i(x)\bigr) \\
&&{} + m \bigl( f_k(x_-)\,\partial_t\delta_0 - f_k(x_+)\,\partial
_t\delta_T \bigr)
\nonumber
\end{eqnarray}
for $k=1, \ldots, d$ where we used Einstein's summation convention
over repeated indices, $\delta_0$ and $\delta_T$ are the Dirac
distributions at $0$ and $T$, respectively, and all derivatives are
taken in the distributional sense. Assume that $f\in\rC^2(\R^d,
\R^d)$, that the partial derivatives $\partial_i f$ and $\partial
_{ij} f$ are
bounded and globally Lipschitz continuous for all $i,j=1,\ldots, d$,
and that there are constants $\beta<1$ and $c>0$ such that $|f(x)|
\leq|x|^\beta+c$ for all $x\in\R^d$. Furthermore assume that the
SDE (\ref{E:target}) a.s. has a solution up to time $T$. Then the
following statements hold:
\begin{enumerate}[(a)]
\item[(a)] For every $x_0\in\rL^2([0,T],\R^d)$,
equation (\ref{E:SPDE}) has a unique, global, continuous mild
solution with $\E(\|x(\tau)\|_{\rL^2}^2) < \infty$ for all
$\tau>0$.
\item[(b)] For $\tau> 0$, the solution $x(\tau)$ a.s. takes values
in the Sobolev space $\rH^1([0,T]$, $\R^d)$.
\item[(c)] The distribution $Q_f^{0,x_-;T,x_+}$ given by
Definition \ref{D:target} is invariant for (\ref{E:SPDE}).
\end{enumerate}
\end{theorem}
\begin{remark}
The sub-linear growth-condition $|f(x)| \leq|x|^\beta+c$ on the
drift seems quite technical. The condition is only required for the
bounds in Lemma \ref{L:U-bound}. We believe that an additional,
linear drift term can be added by incorporating it into the linear
operator $\CL$, following a similar procedure in \cite
{HairerStuartVossWiberg05}.
\end{remark}
\begin{remark}
In \cite{HairerStuartVoss07}, Remark 5.5, we see that the terms
involving derivatives of Dirac distributions can be interpreted as
modifications to the boundary conditions. Proceeding this way, we
see that (\ref{E:SPDE}) is formally equivalent to the SPDE
\begin{eqnarray*}
\partial_\tau x_k(\tau, t)
&=& L x_k (\tau, t)
- f_i(x) \,\partial_k f_i(x)
+ m \,\partial_t x_i\, \partial_t x_j \,\partial_{ij}^2 f_k(x) \\
&&{} - \partial_t x_i \bigl(\partial_i f_k(x) - \partial_k
f_i(x)\bigr)
+ m \,\partial_t^2 x_j \bigl(\partial_j f_k(x) + \partial_k
f_j(x)\bigr)\\
&&{} + \sqrt{2} \,\partial_\tau w_k(\tau, t),
\end{eqnarray*}
where $\partial_\tau w$ is space--time white noise, endowed with the
boundary conditions
\begin{eqnarray*}
x(0) &=& x_-,\qquad
x(T) = x_+, \\
m \,\partial_t^2 x(0) &=& \partial_t x(0) + f(x_-),\qquad
m \,\partial_t^2 x(T) = - \partial_t x(T) + f(x_+).
\end{eqnarray*}

In the one-dimensional case $f\dvtx\R\to\R$ this SPDE simplifies further
to
\[
\partial_\tau x(\tau, t)
= L x(\tau, t)
- f(x) f'(x)
+ m (\partial_t x)^2 f''(x)
+ 2m \,\partial_t^2 x f'(x)
+ \sqrt{2} \,\partial_\tau w(\tau, t).
\]
\end{remark}
\begin{remark}
Using a standard bootstrapping argument like the one in the proof of
\cite{Hairer08}, Theorem 6.5, one can show that, in fact, the
solution $x$ of (\ref{E:SPDE}) takes values in the Sobolev space
$\rH^r([0,T], \R^d)$ for every $r<3/2$.
\end{remark}

The remainder of the article gives a proof of
Theorem \ref{T:answer}. We start the argument, in
Section \ref{S:operator} by collecting some results about the
differential operator $\CL$. Section~\ref{S:linear} shows that the
theorem holds for the linear case $f\equiv0$, in which case
$\CN\equiv0$. Finally, Section \ref{S:nonlinear} completes the
proof by showing that introduction of the drift $\CN$ changes the
stationary distribution of (\ref{E:SPDE}) in the correct way to
account for a nonvanishing $f$.

%%%%%%%%%%%%%%%%%%%%%%%%%%%%%%%%%%%%%%%%%%%%%%%%%%%%%%%%%%%%%%%%%%%%%%

%s3 ###
\section{Analysis of the linear operator}
\label{S:operator}

This section collects some results about the operator $\CL$ from
Definition \ref{def:L-and-mean}. Since we are only interested in the
operator itself and not in the full SPDE, in addition to the
scaling-argument from Remark \ref{R:rescaling}, we can rescale $t$.
Thus, throughout Section \ref{S:operator}, we will consider the
operator $\bar\CL$ defined as
\[
\bar\CL= - \partial_t^4 + \gamma^2 \,\partial_t^2,
\]
where $\gamma= \frac{T}{\pi m}$, on the domain
\begin{eqnarray*}
\CD(\bar\CL) &=& \biggl\{ x\in\rH^4 \bigm|
x(0)=x(\pi)=0,\\
&&\hspace*{4.8pt}
\frac{T}{\pi\gamma} \,\partial_t^2 x(0) = \partial_t x(0),
\frac{T}{\pi\gamma} \,\partial_t^2 x(\pi) = -\partial_t x(\pi)
\biggr\}.
\end{eqnarray*}
Then, after rescaling $t$, $\bar\CL$ differs from the operator
$\CL$ from Definition \ref{def:L-and-mean} only by multiplication of a
positive constant.

Throughout the rest of the paper we will use the following notation:
we denote by $[S(\tau)]_{\tau\geq0}$ the semigroup associated to
$\bar\CL$ on $\CH=\rL^2([0,\pi], \R^d)$ and by $\CH
_\alpha=
\CD[(-\bar\CL)^\alpha]$ the associated interpolation spaces.

%s3.1 ###
\subsection{Approximation to the spectral decomposition}

\begin{lemma}
$\bar\CL$ is a self-adjoint, negative definite operator on
$\rL^2([0,\pi], \R^d)$.
\end{lemma}
\begin{pf}
Using partial integration it is easy to see that
\begin{eqnarray*}
\langle x, \CL y\rangle
&=& - m^2 \int_0^\pi\partial_t^2 x \,\partial_t^2 y \,dt
- \int_0^\pi\partial_t x \,\partial_t y \,dt \\
&&{}
- m \bigl( \partial_t x(0) \,\partial_t y(0) + \partial_t x(\pi)\,
\partial_t y(\pi) \bigr)
\end{eqnarray*}
for all $x,y\in\CD(\CL)$, that is, the operator $\CL$ is
symmetric and negative. Its self-adjointness can be checked in \cite
{ReeSim72}, Section VIII.
\end{pf}
\begin{lemma}\label{L:eigenvalues}
Let $\lambda_k$, $k\in\N$ be the eigenvalues of $-\bar\CL$ and
$e_k$ be
the corresponding eigenfunctions. Define, furthermore,
\[
\bigl( f_k^{(i)} \mid i=1,2,3,4 \bigr)
= \bigl( \sin kt, \cos kt, e^{-kt}, e^{-k(\pi-t)} \bigr).
\]
Then the following statements hold:
\begin{enumerate}[(a)]
\item[(a)] The eigenvalues of $\bar\CL$ satisfy $\lambda_k = k^4 +
\CO(k^2)$.
\item[(b)] There exist functions $g_k^{(i)}$ such that $e_k(t) = \sin
(kt) + {1\over k}\sum_{j=1}^4 g_k^{(i)}(t) f_k^{(i)}(t)$ for all
$t\in[0,\pi]$ and such that ${\sup_{j=1}^4\sup_{k\in\N}}
\|g_k^{(j)}\|_{\rC^j} < \infty$ for every $j\ge0$.
\end{enumerate}
\end{lemma}
\begin{pf}
Since $\bar\CL$ acts independently on each coordinate, we can assume
$d=1$ without loss of generality. The eigenfunctions of $\bar\CL$ can be
written in the form
\[
x(t) = \xi_1 \rmE^{\kappa_+ (t-\pi)} + \xi_2 \rmE^{-\kappa_+ t}
+ \xi_3 \rmE^{i \kappa_- t} + \xi_4 \rmE^{-i \kappa_- t},
\]
where
\[
\kappa_\pm= \sqrt{\sqrt{\mu^4 + {\gamma^4 \over4}} \pm{\gamma
^2 \over2}} = \mu\pm{\gamma^2 \over4\mu} + \CO(1/\mu^3),
\]
with $\lambda= \mu^4$ the corresponding
eigenvalue. The coefficient vector $\xi\in\C^4$ is determined by
the boundary conditions: for $x$ to be an eigenfunction of $\bar\CL$,
$\xi$ must
satisfy $A_\mu\xi= 0$ where
\begin{eqnarray*}
A_\mu&=& \left(\matrix{
\rmE^{-\kappa_+\pi}                     & 1                                     \cr
1                                       & \rmE^{-\kappa_+\pi}                   \cr
(\alpha\kappa_+-1)\rmE^{-\kappa_+\pi}   & \alpha\kappa_+ +1                     \cr
\alpha\kappa_++1                        & (\alpha\kappa_+-1)\rmE^{-\kappa_+\pi}}\right.
\\
&&\hspace*{-3.6pt}
\left.\matrix{
& 1                                         & 1                                        \cr
& \rmE^{i\kappa_-\pi}                       & \rmE^{-i\kappa_-\pi}                     \cr
& -\alpha\kappa_--i                         & -\alpha\kappa_-+i                        \cr
& (-\alpha\kappa_-+i)\rmE^{i\kappa_-\pi}    &(-\alpha\kappa_--i)\rmE^{-i\kappa_-\pi}}\right)
\end{eqnarray*}
and $\alpha= T / \pi\gamma$. Setting $\kappa= (\kappa_+ +
\kappa_-)/2 = \mu+ \CO(1/\mu^3)$ for ease of notation, we note that
this equation has nonzero solutions if and only if
\begin{eqnarray*}
0 &=& \det A_\mu= 8i \bigl((\alpha^2 \kappa^2 + \alpha\kappa
)\sin(\kappa_-\pi) - \bigl( \tfrac{1}{2}+\alpha\kappa
\bigr)\cos(\kappa_-\pi)\bigr) + \CO(1/\mu)\\
&=& 8i \alpha^2 \mu^2\sin\mu\pi+ \CO(\mu).
\end{eqnarray*}
It follows immediately that, at least for large values of $\mu$, one
has $\mu= k + \CO(1/k)$ with $k \in\N$ so that $\lambda_k = k^4 +
\CO(k^2)$ as requested. In particular, one has
\[
\kappa_{\pm} = k + {\beta_\pm\over k} + \CO(1/k^2),
\]
for some constants $\beta_\pm\in\R$.
It remains to check the statement about the eigenfunctions.

Given that we already have good control on the eigenvalues, our claim
will follow if we are able to show that one can choose $\xi=
(0,0,{1\over2}, -{1\over2}) + \CO(1/k)$. Expanding $A_\mu$ in powers
of $k$, we obtain
\begin{eqnarray*}
A_\mu&=& k
\pmatrix{
0 & 0 & 0 & 0 \cr
0 & 0 & 0 & 0 \cr
0 & \alpha& -\alpha& -\alpha\cr
\alpha& 0 & -\alpha& -\alpha}
+
\pmatrix{
0 & 1 & 1 & 1 \cr
1 & 0 & (-1)^k & (-1)^k \cr
0 & 1 & -i & i \cr
1 & 0 & ic (-1)^k & -ic(-1)^k}
+ \CO(1/k) \\
&\equiv& k A_\mu^{(0)} + A_\mu^{(1)} + \CO(1/k),
\end{eqnarray*}
for $c = 1 - \alpha\beta_- \pi$. It now follows from standard
perturbation theory (see, e.g., \cite{Kato66}, Theorem II.5.4)
that the eigenvector $\xi$ with eigenvalue $0$ can be written as $\xi
= \xi^{(0)} + \CO(1/k)$, where $\xi^{(0)}$ satisfies $A_\mu^{(0)}
\xi^{(0)} = 0$. Since $A_\mu^{(0)}$ is degenerate, this is, however, not
sufficient to determine $\xi^{(0)}$ uniquely but only tells us that
$\xi^{(0)}$ is of the form $(a+b,a+b,a,b)$ for $a,b\in\R$. In order
to determine $a$ and~$b$, we have to consider the next order which
yields the compatibility condition $A_\mu^{(1)} \xi^{(0)} \in\range
A_\mu^{(0)}$. This compatibility condition can be rewritten as $a+b =
0$, so that we can indeed choose $\xi^{(0)} = (0,0,{1\over2},
-{1\over2})$, as requested.
\end{pf}

%s3.2 ###
\subsection{The relation between interpolation and Sobolev spaces}
\label{S:sobolev}

In this section, we show how the interpolation spaces $\CH_\alpha$
associated to the operator $\bar\CL$ relate to the usual fractional
Sobolev spaces. These results are ``well known'' in the folklore of
the subject. However, in our specific context (especially since we
need to consider fractional exponents), we were not able to derive
them as straightforward corollaries from results in standard textbooks
on function spaces, like \cite{Trieb1,Trieb2,Trieb3}. Because of
this, and since one can find rather short and self-contained proofs,
we prefer to include them here.

Before we turn to this however, we start with a comparison between the
interpolation spaces of the Dirichlet Laplacian and the periodic
Laplacian. These are going to be useful in the sequel.

Let $\Delta_0$ denote the Laplacian on $[0,\pi]$ with Dirichlet
boundary conditions and let $\Delta$ denote the Laplacian on
$[0,2\pi]$ with periodic boundary conditions. These operators are
self-adjoint in $H_0 = \rL^2([0,\pi])$ and $H = \rL
^2([0,2\pi])$,
respectively. We denote by $\rH^s_0$ the domain of $\Delta_0^{s/2}$ and
by $\rH^s$ the domain of $\Delta^{s/2}$ (defined in the usual way
through spectral decomposition). The aim of this section is to study
the correspondence between these two different types of fractional
Sobolev spaces. Denote by $\iota\dvtx\rH^s_0 \to\rH^s$ the map
\[
\iota f(t)
= \cases{
f(t), &\quad for $t \in[0,\pi]$,\cr
-f(2\pi-t), &\quad for $t \in(\pi, 2\pi]$.}
\]
Note that $\iota/ 2$ is an isometry since it maps the eigenfunctions
of $\Delta_0$ into eigenfunctions of $\Delta$. This, therefore, defines
an inclusion $\rH^s_0 \subseteq\rH^s$. A natural left inverse for
$\iota$ is given by the restriction map
\[
rf = f |_{[0,\pi]} \qquad\mbox{for all $f\in\rH^s$.}
\]
However, $r$ is not an isometry and, for $s \ge1/2$, it certainly
does not map $\rH^s$ into $\rH^s_0$ in general (since the constant function
$1$ belongs to every $\rH^s$ but only belongs to $\rH^s_0$ for $s < 1/2$).
We do, however, have the following:
\begin{lemma}\label{lem:restr}
The restriction operator $r$ is bounded from $\rH^s$ into $\rH^s_0$
for any $s
< 1/2$.
\end{lemma}
\begin{pf}
First note that $\rH^s$ is isomorphic to $H$ via the isomorphism $x
\mapsto\Delta^{s/2}x$ and similarly for $\rH^s_0$ so that the study
of $r$ as an operator from $\rH^s$ to $\rH^s_0$ is equivalent to the
study of the operator $\Delta_0^{ s/2}r\Delta^{-s/2}$ from $H$ to
$H_0$. Furthermore, we know that $r$ is bounded from $H$ to $H_0$
so that it suffices to show that the operator $A =
(\Delta_0^{s/2}r\Delta^{-s/2} - r)$ is bounded from $H$ to
$H_0$.

Since $A$ maps $\sin(n\quark)$ to $0$ for every $n$, it suffices to
consider $A$ on the subspace of $H$ given by even functions and
generated by the basis of eigenfunctions of $\Delta$ given by
$\varphi_n(t) = {1\over\pi}\cos nt$. Define, furthermore, the basis of
eigenfunctions of $\Delta_0$ given by $\psi_m = {2\over\pi}\sin
mt$. This yields for $A$ the matrix elements
\begin{eqnarray*}
A_{mn} &=& \langle\psi_m, A\varphi_n \rangle= {2\over\pi^2} (m^s
n^{-s} -
1)\int_0^\pi\sin(mt) \cos(nt) \,dt \\
&=& \cases{
\dfrac{4m(m^s n^{-s} -1) }{ \pi^2 (m+n)(m-n)}, &\quad if $m+n$ is odd,\vspace*{2pt}\cr
0, &\quad if $m+n$ is even.}
\end{eqnarray*}
It follows that there exists a constant $C>0$ such that $\|A\| \le
C\|\hat A\|$, where the operator $\hat A$ is defined via its matrix
elements by
\[
\hat A_{mn} = \cases{{m n^{-2}}, &\quad if $n \ge m$,\cr
m^{s-1}n^{-s}, &\quad if $m \ge n$.}
\]
Now it is a straightforward exercise in linear algebra to show that,
given an orthonormal basis $\{\varphi_n\}_{n \ge0}$, an operator
$\hat
A$ is bounded if there exists positive numbers $f_{m,n}$ such that the
inequalities
%
%e11 ###
%
\begin{equation}\label{e:boundAhat}
\sup_{n\ge0} \sum_{m\ge0} {|\langle\hat A\varphi_n, \hat
A\varphi_m \rangle|\over f_{n,m}} < \infty,\qquad \sup_{m\ge0}
\sum_{n\ge0}
|\langle\hat A\varphi_n, \hat A\varphi_m \rangle| f_{n,m} < \infty
\end{equation}
both hold. (Just expand $\|\hat A x\|^2$ for $x = \sum_{n \ge0} x_n
\varphi_n$ and make use of the inequality $|x_n x_m| \le{x_n^2\over
2f_{n,m}} + {x_m^2 f_{n,m}\over2}$.) We will show that
(\ref{e:boundAhat}) does indeed hold for $\hat A$ as above.
Assuming without loss of generality that $m \ge n$, we have the
bound
\begin{eqnarray*}
|\langle\hat A\varphi_n, \hat A\varphi_m \rangle| &\le& \sum
_{k=1}^n k^2
m^{-2} n^{-2} + \sum_{k=n}^m k^s m^{-2} n^{-s} + \sum_{k=m}^{\infty}
k^{2s-2} m^{-s} n^{-s} \\
&\le& C (n m^{-2} + m^{s-1}n^{-s}) \le C m^{s-1}n^{-s}
\end{eqnarray*}
and similarly for $n \ge m$. Here we have made use of the fact
that $s < {1\over2}$ to ensure that the last sum converges. It
remains to check that the bounds (\ref{e:boundAhat}) are satisfied
for some choice of $f_{m,n}$. With the choice $f_{m,n} =
\sqrt{m/n}$, we obtain
\[
\sum_{m\ge0} {|\langle\hat A\varphi_n, \hat A\varphi_m \rangle
|\over
f_{n,m}} \le C\sum_{m=1}^n m^{-s-{1/2}}n^{s-{1/2}} + C\sum
_{m=n}^\infty m^{s-{3/2}}n^{{1/2}-s} \le C,
\]
where we made again use of the fact that $s < {1\over2}$. The second
bound in (\ref{e:boundAhat}) is obtained in an identical way with the
roles of $m$ and $n$ reversed.
\end{pf}

For $s > 1/2$, the problem is that elements of $\rH^s_0$ are forced to
be equal to $0$ at the boundary, which is not the case for elements of
$\rH^s$. One has, however, the following:
\begin{lemma}\label{lem:vanish}
For any $s \in(1/2,2]$, the map $r$ is
bounded from the subspace of $\rH^s$ consisting of functions that
vanish at $0$ and $\pi$ into $\rH^s_0$.
\end{lemma}
\begin{pf}
Instead of considering the restriction operator $r$ as before, we
are going to consider the operator $\tilde r$ defined on continuous
functions as
\[
(\tilde r f)(t) = f(t) - {1\over\pi} \bigl(f(0)(\pi-t)
+ f(1)t\bigr).
\]
Note that $\tilde r f = rf$ if $f(0) = f(\pi) = 0$, so that the
statement will be implied by the fact that $\tilde r$ is shown to be
a bounded operator from $\rH^s$ to $\rH^s_0$. Therefore,
instead of considering $A$ as before, we consider the operator
$\tilde A = (\Delta_0^{s/2}\tilde
r\Delta^{-s/2} - r)$ which has matrix elements
\begin{eqnarray*}
\tilde A_{mn} &=& \cases{ A_{mn} - \dfrac{4 }{\pi^2}m^{s-1}n^{-s},
&\quad
$m+n$ odd,\vspace*{2pt}\cr
0, &\quad $m+n$ even.} \\
&=& \cases{ \dfrac{4 }{\pi^2}\dfrac{m}{m^2 - n^2} (m^{s-2}n^{2-s} -
1),
&\quad $m+n$ odd,\vspace*{2pt}\cr
0, &\quad $m+n$ even.}
\end{eqnarray*}
Note that $s = 2$ is a special case since one then has $\tilde A = 0$
as a consequence of the relation $\Delta_0 \tilde r = r \Delta$.

In this regime, we have as before $\|\tilde A\| \le C\|\hat A\|$, but
this time $\hat A$ is defined via its matrix elements by
\[
\hat A_{mn} = \cases{ {m^{s-1} n^{-s}}, &\quad if $n \ge m$,\cr
m^{-1}, &\quad if $m \ge n$.}
\]
Computing $\langle\hat A\varphi_m, \hat A\varphi_n \rangle$ for $m
\ge n$ as
before, we note that there is a difference between the case $s \le1$
and the case $s \ge1$.
We obtain
\[
|\langle\hat A\varphi_m, \hat A\varphi_n \rangle| \le\cases
{m^{-1}, &\quad $s
\ge1$,\cr
n^{s-1}m^{-s}, &\quad $s < 1$.}
\]
For $s < 1$, we now make the choice $f_{m,n} = m^{s-\varepsilon
}n^{\varepsilon-s}$, where $\varepsilon>0$ is chosen sufficiently
small so that $2s-\varepsilon> 1$
(this is always possible since $s > {1\over2}$). With this choice, we obtain
\[
\sum_{m\ge0} {|\langle\hat A\varphi_n, \hat A\varphi_m \rangle
|\over
f_{n,m}} \le C\sum_{m=1}^n m^{\varepsilon-1}n^{-\varepsilon} + C\sum
_{m=n}^\infty m^{\varepsilon-2s}n^{2s-1-\varepsilon} \le C
\]
and similarly for the other term. This calculation also works for the
case $s = 1$ so that the case $s \ge1$ can be obtained in an identical
manner (set, e.g., $\varepsilon= {1\over2}$).
\end{pf}

Consider now the operator $\CL_a$ given by $(\CL_a f)(t) =
\partial_t^4 f$, endowed with the boundary conditions $f(0) = f(\pi)
= 0$
and $f''(0) = -a f'(0)$, $f''(\pi) = a f'(\pi)$.
Since the domain of the square of the Dirichlet Laplacian is
$\CD(\Delta_0^2) = \{f\in\CD(\Delta_0) \mid \Delta_0 f
\in
\CD(\Delta_0)\} = \CD(\CL_0)$, we have $\CL_0 = \Delta
_0^2$. The
following lemma shows that $\CL_a$ for $a\neq0$ can still be viewed
as a perturbation of $\Delta_0^2$:
\begin{proposition}\label{prop:interp}
Fix $a\in\R$ and $\varepsilon> 0$ be arbitrary and define the linear
operator $A\dvtx\rH_0^{{3/2} + \varepsilon} \to\rH
_0^{-{3/2} -
\varepsilon}$ by
\[
A f = f'(0) \delta'_0 - f'(\pi) \delta'_\pi.
\]
Then, the operator $\tilde\CL_a = \Delta_0^2 + a A$ is the
generator of an analytic semigroup on~$H_0$. Furthermore,\vspace*{1pt} this
semigroup coincides with the one generated by $\CL_a$ so that
$\tilde\CL_a = \CL_a$. As a consequence, we obtain the identities
$\CH^\alpha= \rH_0^{4\alpha}$ for every $\alpha\in(-{5\over8},
{5\over8})$.
\end{proposition}
\begin{pf}
First, note that $A$ is well defined since it follows from standard
Sobolev embedding theorems that $f'$ is continuous for every $f \in
\rH^{{3/2} + \varepsilon}$ and, therefore, for every $f \in
\rH_0^{{3/2} + \varepsilon}$. Thus, $\delta_0'$ and $\delta
_\pi'$ can
be considered as elements of the Sobolev
space $\rH_0^{-(3/2+\varepsilon)}$.
Since\vspace*{1pt} $\Delta_0^2$ generates an analytic semigroup
on $\rH_0^\alpha$ for any $\alpha\in\R$, it follows from applying
\cite{Hairer08}, Proposition 4.42, once with $\CB= \rH_0^{-{3/
2}}$ and once with $\CB= \rH_0^{-{5/2}+\varepsilon}$, that
$\tilde
\CL_a$ is the generator\vspace*{1pt} of an analytic semigroup on $\rH_0^\alpha$ for
every $\alpha\in(-{5\over2}, {5\over2})$ and that the
corresponding scale of interpolation spaces satisfies $\tilde
\CH^\alpha= \rH_0^{4\alpha}$ for every $\alpha\in(-{5\over8},
{5\over8})$.

It, therefore, remains to show that the semigroup $\tilde S_\tau$
generated by $\tilde\CL_a$ coincides with the semigroup $S_\tau$
generated by $\CL_a$. Since, for any $u \in\rH_0$, any $\tau> 0$
and any $t \in(0,\pi)$, we have the identities
\[
\partial_\tau\tilde S u (\tau,t) = -\partial_t^4 \tilde S u (\tau
,t) ,\qquad
\partial_\tau S u (\tau,t) = -\partial_t^4 S u (\tau,t) ,
\]
it suffices to show that $\tilde S_\tau u \in\CD(\CL_a)$ for $\tau
> 0$. Since we already know that $\tilde\CH^{1/4} = \rH_0^1$, for
example, we have $(\tilde S_\tau u)(0) = (\tilde S_\tau u)(\pi) =
0$ so that only the second set of boundary conditions needs to be
checked. For this, writing $S_\tau^0$ for the semigroup generated
by $\Delta_0^2$, note that we have the identity
%
%e12 ###
%
\begin{equation}\label{e:tildeS}
\tilde S_\tau u
= S_\tau^0 u
+ a \int_0^\tau S_{\tau- r}^0 A \tilde S_\tau u \,dr
+ a \int_0^\tau S_{\tau- r}^0 A (\tilde S_r u - \tilde S_\tau
u)\, dr.
\end{equation}
Therefore, the first term in this equation belongs to $\rH_0^4$.
Furthermore, it follows from the definition of $A$ that the $\rH_0^4$
norm of the third term is bounded by $C \int_0^\tau(\tau-
r)^{-1-{3/8} - \varepsilon} \|\tilde S_r u - \tilde S_\tau
u\|_{\rH_0^2} \,dr$. Since we know already that $\rH_0^2 = \tilde
\CH^{1/2}$ and since $\tilde S_\tau u \in\tilde\CH^\alpha$
for every $\alpha> 0$, it follows\vspace*{2pt} from standard analytic semigroup
theory that $\|\tilde S_r u - \tilde S_\tau u\|_{\rH_0^2} \le C
|r-\tau|$. So that the third term in (\ref{e:tildeS}) also belongs
to $\rH_0^4$, the second term can be rewritten as
\[
\int_0^\tau S_{\tau- r}^0 A \tilde S_\tau u \,dr = - \Delta_0^{-2}
A \tilde S_\tau u - S_\tau^0 \int_0^\infty S_r^0 A \tilde S_\tau u\,
dr .
\]
Collecting all of this, we conclude that we can write
\[
\tilde S_\tau u = -a \Delta_0^{-2} A \tilde S_\tau u + R_\tau u,
\]
where $R_\tau u \in\rH_0^4$. On the other hand, using an
approximation argument, one can check that if $f \in\rC^1$, then $g
= \Delta_0^{-2} A f$ satisfies the boundary conditions $g''(0) =
f'(0)$ and $g''(\pi) = -f'(\pi)$, from which the claim follows at
once.
\end{pf}
\begin{corollary}\label{cor:inclinterp}
For every $\alpha\in(-{1\over8}, {1\over8})$, we have the
identity $\CH_\alpha= \rH^{4\alpha}([0$, $\pi],\R^d)$.
For every
$\alpha\in({1\over8}, {1\over2}]$, we have the identity
$\CH_\alpha= \rH^{4\alpha}([0,\pi],\R^d)\cap\rC
_0([0,\pi],\R^d)$,
where $\rC_0([0,\pi],\R^d)$ denotes the set of
continuous functions
vanishing at their endpoints. For every $\alpha\in[-{1\over2},
-{1\over8})$, we have $\CH_\alpha= \rH^{4\alpha}([0,\pi]$, $\R
^d) /
\sim$, where the relation $\sim$ identifies distributions that
differ only by a linear combination of $\delta_0$ and $\delta_\pi$.
\end{corollary}
\begin{pf}
By Proposition\vspace*{1pt} \ref{prop:interp} we already know that $\CH_\alpha=
\rH_0^{4\alpha}$ for $\alpha\in[0,{1\over2}]$. The claim for
$\alpha< {1\over8}$ then follows from Lemma \ref{lem:restr} while
the claim for $\alpha\in({1\over8}, {1\over2})$ follows from
Lemma \ref{lem:vanish}. The remaining claims follow from duality.
\end{pf}

%s3.3 ###
\subsection{Well-behaved projection operators}

We will later identify the stationary distribution of the
SPDE (\ref{E:SPDE}) by using a finite-dimensional approximation
argument. When projecting the equation to a finite-dimensional
subspace, the most natural choice of a projection would be to use the
orthogonal projection $\Pi_n$ onto the space spanned by the first $n$
eigenfunctions of $\bar\CL$, but it transpires that these projections
do not possess enough regularity. Instead, we will need to use, in
some places, the operators $\hat\Pi_n$ given by
%
%e13 ###
%
\begin{equation}\label{e:Pi-n}
\hat\Pi_n x = \sum_{k=1}^n \frac{n-k}{n} \langle x, e_k\rangle e_k,
\end{equation}
where the $e_k$ are the eigenfunctions of $\bar\CL$. The purpose of this
section is to prove the required regularity properties
for $\hat\Pi_n$.

We use H\"older norms
\[
\|x\|_{\rC^{1+\alpha}}
= \cases{
\displaystyle \|x\|_\infty+ \|\dot x\|_\infty
+ \sup_{s\neq t}\frac{|\dot x(t) - \dot x(s)|}{|t-s|^\alpha},
&\quad if $x\in\rC^1$ and\vspace*{2pt}\cr
+\infty, &\quad else,}
\]
where $\alpha\in[0,1)$ and write $\rC^{1+\alpha} = \{ x\in
\rC^1
\mid \|x\|_{\rC^{1+\alpha}} < \infty\}$ and $\rC^{1+\alpha}_0
= \{ x\in\rC^{1+\alpha} \mid x(0) = x(\pi) = 0 \}$.
\begin{lemma}
Let $f_k\dvtx[0,\pi]\to\R$ be defined by $f_k(t) = \sin(kt)$.
Define the operators
$\hat\Pi^0_n x = \sum_{k=1}^n \frac{n-k}{n} \langle x, f_k\rangle f_k$.

\begin{enumerate}[(a)]
\item[(a)] Let $F_n$ be the Fej\'er kernel given by $F_n(t) = \frac{1}{n}
[\sin(\frac{nt}{2})/\sin(\frac{t}{2})]^2$ for all
$t\in[-\pi,\pi]$. Then $\hat\Pi^0_n x = -{1\over4} F_n * \tilde
x$, where $\tilde x$ is the antisymmetric continuation
of~$x$.\vspace*{1pt}

\item[(b)] $\|\hat\Pi_n^0 x \|_{\rC^{1+\alpha}} \leq2\pi\|x\|_{\rC
^{1+\alpha}}$ for
all $x\in\rC^{1+\alpha}$ and all $\alpha\in(0,1)$.
\end{enumerate}
\end{lemma}
\begin{pf}
(a) Since $\int_{-\pi}^\pi\tilde x(t) \sin(kt) \,dt = 2\langle x,
f_k \rangle$ and $\int_{-\pi}^\pi\tilde x(t) \cos(kt) \,dt = 0$, it
follows from trigonometric identities that
\[
\langle x, f_k \rangle f_k(s)
= -{1\over2}\int_{-\pi}^\pi\tilde x(t) \cos\bigl(k(s-t)\bigr)\, dt.
\]
The result then follows from the fact that $F_n(t) = 2 \sum_{k=1}^n
\frac{n-k}{n} \cos kt$.

(b) This follows directly from part (a) using the definition of the
$\rC^{1+\alpha}$-norm and properties of the convolution operator.
\end{pf}
\begin{lemma}\label{lem:boundsekxk}
Let $\alpha\in(0,1)$ and let $x\in\rC^{1+\alpha}$ with $x(0) =
x(\pi) = 0$. Then there exists a constant $c>0$ such that the bounds:
\begin{enumerate}[(2)]
\item[(1)] $\|f_k\|_{\rC^{1+\alpha}} \le c k^{1+\alpha}$,
\item[(2)] $\|e_k - f_k\|_{\rC^{1+\alpha}} \le c k^{\alpha}$,
\item[(3)] $|\langle x,f_k \rangle| \le c\|x\|_{\rC^{1+\alpha}}
k^{-1-\alpha
}$ and
\item[(4)] $|\langle x, e_k - f_k \rangle| \le c \|x\|_{\rC
^{1+\alpha}}
k^{-2-\alpha}$
\end{enumerate}
hold for every $k\in\N$.
\end{lemma}
\begin{pf}
The first bound is standard. The second bound follows immediately
from Lemma \ref{L:eigenvalues}, part (b).
For the third bound, we use partial integration to get
\[
\langle x, f_k\rangle
= \frac{1}{k} \int_0^\pi\dot x(t) \cos(kt) \,dt
= \frac{1}{k} \sum_{j=0}^{k-1} (-1)^j
\int_0^{\pi/k} \dot x\biggl(t+\frac{j}{k}\pi\biggr) \cos(ks) \,ds.
\]
Writing $|\dot x|_\alpha= \sup_{s\neq t}\frac{|\dot x(t) - \dot
x(s)|}{|t-s|^\alpha}$, it is easy to see that each term of the sum
is of order $\CO(|\dot x|_\alpha/ k^{1+\alpha})$ and the
claim follows from this.

The bound on $\langle x,e_k-f_k \rangle$ follows similarly: if $g$ is any
$\rC^{1+\alpha}$ function with $g(0) = g(\pi) = 0$, we can use
integration by parts to get
\begin{eqnarray*}
\int_0^\pi g(t) \sin(kt) \,dt
&=& {1\over k} \int_0^\pi\dot g(t) \cos(kt) \,dt ,\\
\int_0^\pi g(t) e^{-kt} \,dt
&=& {1\over k} \int_0^\pi\dot g(t) e^{-kt} \,dt
\end{eqnarray*}
and similar results for integrals against $\cos kt$ and
$e^{-k(\pi-t)}$. As above, these expressions are bounded by
$\CO(k^{-1-\alpha})$. The claim now follows from
Lemma \ref{L:eigenvalues}, part~(b), by absorbing the slowly varying
terms $g_k^{(j)}$ into $g$.
\end{pf}

The following lemma collects all the properties we will require for
the operators~$\hat\Pi_n$. These will be used in the proof
of Proposition \ref{P:stationary} below.
\begin{lemma}\label{L:pi-n-hat}
Let $(e_n)_{n\in\N}$ be an orthonormal system of eigenfunctions
of $\CL$ and denote by $\Pi_n$ the orthogonal projection of $\CH$
onto $E_n = \vspan\{ e_1, \ldots, e_n \}$. Define $\hat\Pi_n\dvtx
\CH\to E_n$ as in (\ref{e:Pi-n}). Then the following statements
hold:
\begin{enumerate}[(a)]
\item[(a)] $\hat\Pi_n\circ\Pi_n = \hat\Pi_n$.
\item[(b)] $\hat\Pi_n x \to x$ in $\CH_\alpha$ as $n\to\infty$
for all $\alpha\in\R$.
\item[(c)] $\|\hat\Pi_n\|_{\CH_\alpha} \leq1$ for all $n\in\N$ and
$\alpha\in\R$.
\item[(d)] For every $0 < \alpha< \beta< 1$ we have $\| \hat\Pi_n
\|_{\rC^{1+\beta}_0 \to\rC^{1+\alpha}_0} < \infty$.
\item[(e)] Let $0 < \alpha< \beta< 1/2$ and $x\in\rC^{1+\beta}_0$.
Then $\|\hat\Pi_n x - x\|_{\rC^{1+\alpha}} \to0$ as $n\to\infty$.
\end{enumerate}
\end{lemma}
\begin{pf}
Statement (a) is clear from the definition of $\hat\Pi_n$. Let $x =
\sum_{k=1}^\infty x_k e_k\break\in\CH_\alpha$. Then
\[
x - \hat\Pi_n x
= \sum_{k=1}^n \frac{k}{n} x_k e_k
+ \sum_{k=n+1}^\infty x_k e_k
\]
and thus, writing $\lambda_k$ for the eigenvalues of $-\CL$,
\[
\| x - \hat\Pi_n x \|_{\CH_\alpha}^2
= \| (-\CL)^\alpha(x - \hat\Pi_n x) \|_{\rL^2}^2
= \sum_{k=1}^n \frac{k^2}{n^2} \lambda_k^{2\alpha} x_k^2
+ \sum_{k=n+1}^\infty\lambda_k^{2\alpha} x_k^2
\longrightarrow0
\]
as $n\to\infty$. This proves statement (b). Similarly, we have
\[
\| \hat\Pi_n x \|_{\CH_\alpha}^2
= \sum_{k=1}^n \frac{(n-k)^2}{n^2} \lambda_k^{2\alpha} x_k^2
\leq\sum_{k=1}^\infty\lambda_k^{2\alpha} x_k^2
= \| x \|_{\CH_\alpha}^2,
\]
which is statement (c).

(d) From Lemma \ref{lem:boundsekxk} we get $\|e_k\|_{\rC^{1+\alpha}}
\leq\|f_k\|_{\rC^{1+\alpha}} + \|e_k - f_k\|_{\rC^{1+\alpha}} \leq
c k^{1+\alpha}$. Using this and the other bounds from
Lemma \ref{lem:boundsekxk}, we obtain
\begin{eqnarray*}
\|\hat\Pi^0_n x - \hat\Pi x\|_{\rC^{1+\alpha}} &=& \sum_{k =1}^n
{n-k \over n} \|\langle x,e_k-f_k \rangle e_k + \langle x,f_k
\rangle(e_k-f_k)\|_{\rC^{1+\alpha}} \\
&\le&\sum_{k =1}^\infty(|\langle x,e_k-f_k \rangle
| \|e_k\|
_{\rC^{1+\alpha}} + |\langle x,f_k \rangle| \|e_k-f_k\|
_{\rC
^{1+\alpha}} ) \\
&\le& C\|x\|_{\rC^{1+\beta}} \sum_{k =1}^\infty k^{-1-(\beta-\alpha)}.
\end{eqnarray*}
Since we already know that $\hat\Pi_n^0$ satisfies the requested
bound, the claim follows.

(e) Let $\varepsilon>0$. We can write $x\in\rC^{1+\beta}_0$ as $x =
y + z$
with $\|y\|_{\rC^{1+\alpha}} \leq\varepsilon$ and $z\in\rH^2$
with $z(0) =
z(T) = 0$. This gives
\[
\| \hat\Pi_n x - x \|_{\rC^{1+\alpha}}
\leq c \varepsilon+ \| \hat\Pi_n z - z \|_{\rC
^{1+\alpha}}.
\]
Because $1 + \alpha+ 1/2 < 2$, we have $\|z\|_{\rC^{1+\alpha}} \leq c
\|z\|_{\rH^2}$ for all $z \in\rH^2$. Corollary \ref{cor:inclinterp}
gives $\rH^2 \cap C_0 = \CH_{1/2}$ and thus,
\[
\| \hat\Pi_n x - x \|_{\rC^{1+\alpha}}
\leq c \varepsilon+ \| \hat\Pi_n z - z \|_{\CH_{1/2}}
\to c \varepsilon
\]
as $n\to\infty$ by part (b). Since we can choose $\varepsilon>0$ arbitrarily
small, the proof is complete.
\end{pf}

%%%%%%%%%%%%%%%%%%%%%%%%%%%%%%%%%%%%%%%%%%%%%%%%%%%%%%%%%%%%%%%%%%%%%%

%s4 ###
\section{The linear case}
\label{S:linear}

This section gives the proof of Theorem \ref{T:answer} for the linear
case $f\equiv0$.
\begin{lemma}\label{L:stat-dist-linear}
Let $\CL$ and $\bar x$ be given by Definition \ref{def:L-and-mean}.
Then $Q_0^{0,x_-;T,x_+} = \CN(\bar x$, $-\CL^{-1})$.
\end{lemma}
\begin{pf}
Since for $f=0$ the components of the solution of (\ref{E:target})
are independent, it suffices to work
in dimension $d=1$. First consider the unconditioned process
described by (\ref{E:target}). It is easy to check that $p$
satisfies
\[
p(t) = \rmE^{-t/2m} p(0) + \sqrt{\frac{1}{2}} \int_0^t
\rmE^{-(t-r)/2m} \,dw(r)
\]
and thus,
\[
q(t) = x_- + 2(1-\rmE^{-t/2m}) p(0)
+ \sqrt{\frac{1}{2 m^2}} \int_0^t \int_0^v \rmE^{-(v-r)/2m}\,
dw(r) \,dv.
\]
The mean of this process is
\[
\bar x_0(t) = \E(q(t)) = x_-
\]
and, since $p(0)$ is independent of $w$, the covariance function can
be found as
\begin{eqnarray*}
C_0(s,t)
&=& \Cov(q(s), q(t)) \\
&=& 4 ( 1 - \rmE^{-s/2m} ) ( 1 - \rmE^{-t/2m} ) \frac{m}{2}\\
&&{} + \frac{1}{2 m^2} \int_0^s \int_0^t \int_0^{u\wedge v} \rmE^{-(u+v-2r)/2m} \,dr \,dv \,du
\end{eqnarray*}
for all $s,t\in[0,T]$. Evaluating the integrals and combining the
resulting terms allows us to simplify this to
\[
C_0(s,t)
= 2(s\wedge t)
+ 2m \bigl(
\rmE^{-s/2m} + \rmE^{-t/2m} - \rmE^{-|s-t|/2m} - 1
\bigr).
\]

Denote the mean and covariance function of the process conditioned
on $q(T) = x_+$ by $\bar x$ and $C$, respectively. From
\cite{HairerStuartVossWiberg05}, equations (3.15) and (3.16), we know
that
\[
\bar x(t) = \bar x_0(t) + C_0(t,T) C_0(T,T)^{-1} (x_+-x_-)
\]
and
\[
C(s,t) = C_0(s,t) - C_0(s,T) C_0(T,T)^{-1} C_0(T,t)
\]
for all $s,t \in[0,T]$. The covariance operator
of $Q_0^{0,x_-;T,x_+}$ is then given by
\[
\CC f(s) = \int_0^T C(s,t) f(t) \,dt.
\]

To complete the proof we have to verify that $\bar x$ and $\CC$ have
the required form. The following facts are easily checked:
\begin{enumerate}[(iii)]
\item[(i)] the first derivatives $C_0(s,t)$, $\partial_t C_0(s,t)$ and
$\partial_t^2 C_0(s,t)$ are continuous at $t=s$ and the third derivative
at $t=s$ jumps according to
\[
\partial_t^3 C_0(s, s+)-\partial_t^3 C_0(s, s-) = \frac{1}{2 m^2};
\]
\item[(ii)] the derivative boundary conditions
\[
2m \,\partial_t^2 C_0(s,0) = \partial_t C_0(s,0),\qquad
2m \,\partial_t^2 C_0(s,T) = -\partial_t C_0(s,T)
\]
are satisfied;
\item[(iii)] the left boundary condition
\[
C_0(s,0)=0
\]
holds; and
\item[(iv)] $L C_0(T,t)=0$.
\end{enumerate}

Clearly, by (ii) and (iii), the mean $\bar x$ satisfies the required
boundary conditions (\ref{E:xbar-bc}) and by (iv) it also satisfies
$L\bar x = 0$. From the definition of $L$ and in particular from
properties (i) and (iv), we can deduce
\[
-LC(s,t) = \delta(t-s)
\]
and using (ii) and (iii) we deduce that $C(t,s)$ satisfies the
boundary conditions~(\ref{E:bch}). Thus $C$ is the Green's function
of $-\CL$ and we can deduce that $\mathcal{C}=-\CL^{-1}$ as
required.
\end{pf}
\begin{proposition}\label{P:linear}
Consider the $\rL^2([0,T], \R^d)$-valued equation
%
%e14 ###
%
\begin{equation}\label{E:linear-SPDE}
dy(\tau) = \CL\bigl(y(\tau)-\bar x\bigr)\, d\tau+ \sqrt{2}
\,dw(\tau),\qquad
y(0) = y_0,
\end{equation}
where $y_0\in\rL^2([0,T],\R^d)$.
Then the following statements hold:
\begin{enumerate}[(a)]
\item[(a)] Equation (\ref{E:linear-SPDE}) has a unique, global,
continuous mild solution.
\item[(b)] For every $\alpha<3/8$ the solution $y$ is a.s.
continuous with values in $\CH_\alpha$.
\item[(c)] The distribution $Q_0^{0,x_-;T,x_+}$ is the unique
stationary distribution for (\ref{E:SPDE}).
\end{enumerate}
\end{proposition}
\begin{pf}
From Lemma \ref{L:stat-dist-linear} we know $Q_0^{0,x_-;T,x_+} =
\CN[\bar x, (-\CL)^{-1}]$. Thus, we can
apply \cite{HairerStuartVossWiberg05}, Lemma 2.2, to get that
(\ref{E:linear-SPDE}) has a continuous, $\rL^2([0,T], \R
^d)$-valued
mild solution and $\nu$ is its unique stationary distribution.

Let $\lambda_k$, $k\in\N$ be the eigenvalues of $-\CL$. Then, using
Lemma \ref{L:eigenvalues}, $\tr(-\CL)^{-2\beta} = \sum_{k\in\N}
\lambda_k^{-2\beta} < \infty$ if and only if $\beta>1/8$. Thus, for
example, by applying \cite{Hairer08}, Theorem 5.13, to $y - \bar x$,
the solution takes values in $\CH_\alpha$ for every $\alpha<1/2 -
1/8 = 3/8$ and is continuous by \cite{Hairer08}, Theorem 5.17, (see
also \cite{DaPrato-Zabczyk92} for very similar results). This completes the proof.
\end{pf}

The regularity of the solution given in Proposition \ref{P:linear} is
consistent with the regularity of the target
distribution $Q_0^{0,x_-;T,x_+}$: the process $\dot x$
in (\ref{E:target}) is continuous and lives in $\rH^{1/2-\varepsilon
}$ and
thus $x$ is in $\rH^{3/2-\varepsilon}$ for all $\varepsilon>0$. On
the other hand,
Corollary \ref{cor:inclinterp} shows that $\CH_\alpha\subseteq
\rH^{4\alpha}$ and thus, that $y$ also takes values in $\rH
^{3/2-\varepsilon}$
for all $\varepsilon>0$. The following lemma provides an additional
regularity result for $x$ in stationarity.
\begin{lemma}\label{L:stat-is-Hoelder}
Let $\alpha< 1/2$. Then $x\in C^{1+\alpha}$ for
$Q_0^{0,x_-;T,x_+}$-almost all $x$.
\end{lemma}
\begin{pf}
The result is a direct consequence of
\cite{Hairer08}, Corollary 3.22: let $e_k$ be the eigenfunctions of
$-\CL$ with corresponding eigenvalues $\lambda_k$. By
Lemma \ref{L:stat-dist-linear}, part (b), the random variable
\[
X + \bar x = \sum_{k=1}^\infty\frac{\eta_k}{\sqrt{\lambda_k}} e_k,
\]
where the $\eta_k$ are i.i.d. standard Gaussian random variables,
has distribution $Q_0^{0,x_-;T,x_+}$. We have to show that the
derivative
\[
X' + \bar x' = \sum_{k=1}^\infty\frac{\eta_k}{\sqrt{\lambda_k}} e_k'
\]
is $\alpha$-H\"older continuous.

Let $\delta\in(2\alpha, 1)$. By Lemma \ref{L:eigenvalues} we have
$\|e_k'\|_\infty= \CO(k)$, $\|e_k''\|_\infty= \CO(k^2)$ and
$\lambda_k = c k^4 + \CO(k^2)$ for some $c > 0$. This gives
\[
S_1^2
= \sum_{k=1}^\infty\biggl\| \frac{e_k'}{\sqrt{\lambda_k}} \biggr\|
_\infty^2
< \infty,\qquad
S_2^2
= \sum_{k=1}^\infty\biggl\| \frac{e_k'}{\sqrt{\lambda_k}} \biggr\|
_\infty^{2-\delta}
\Lip\biggl( \frac{e_k'}{\sqrt{\lambda_k}} \biggr)^\delta
\leq\sum_{k=1}^\infty\frac{c}{k^{2-\delta}}
< \infty.
\]
Thus, the conditions of \cite{Hairer08}, Corollary 3.22, are satisfied
and we get the required H\"older continuity.
\end{pf}

%%%%%%%%%%%%%%%%%%%%%%%%%%%%%%%%%%%%%%%%%%%%%%%%%%%%%%%%%%%%%%%%%%%%%%

%s5 ###
\section{The nonlinear case}
\label{S:nonlinear}

In this section we complete the proof of Theorem \ref{T:answer}. The
proof is split in a sequence of results which identify the target
distribution $Q_f^{0,x_-;T,x_+}$, determine the regularity properties
of the drift $\CN$, give existence of global solutions to the
SDE (\ref{E:SPDE}) and, finally, identify the stationary distribution
of this equation.
\begin{lemma}\label{L:density}
Assume that $f\dvtx\R^d\to\R^d$ is such that the
SDE (\ref{E:target}) a.s. has a solution up to time $T$. Let $\mu
= Q_f^{0,x_-;T,x_+}$ and $\nu= Q_0^{0,x_-;T,x_+}$ be the
distributions on $\rL^2([0,T], \R^d)$ from
Definition \ref{D:target}. Then the density $\varphi=
\frac{d\mu}{d\nu}$ is given by
\begin{eqnarray*}
\varphi(x)
&=& \frac{1}{Z} \exp\biggl( m \langle f(x_+),\dot x(T) \rangle- m
\langle f(x_-),\dot x(0)\rangle\\
&&\hspace*{32pt}{}
- \int_0^T
m \langle Df(x(t)) \dot x(t), \dot x(t) \rangle\\
&&\hspace*{33pt}{}
- \langle f(x(t)), \dot x(t)\rangle
+ \frac{1}{2} |f(x(t))|^2\,
dt \biggr),
\end{eqnarray*}
where $Df$ is the Jacobian of $f$ and $Z$ is the required
normalization constant.
\end{lemma}
\begin{pf}
Let $\tilde\mu_{\dot x} = P_f^{0,x_-}$ be the unconditioned distribution
of $\dot x$ in (\ref{E:target}) and let $\tilde\nu_{\dot x} =
P_0^{0,x_-}$ the
same distribution, but for $f=0$. Then the Girsanov formula,
for example, in the form of \cite{HairerStuartVoss07}, Lemma 9,
gives the density of $\tilde\mu_{\dot x}$ w.r.t. $\tilde\nu_{\dot x}$,
\[
\frac{d\tilde\mu_{\dot x}}{d\tilde\nu_{\dot x}}(\dot x)
= \exp\biggl(
\int_0^T \langle f(x(t)), d\dot x(t)\rangle
+ \frac12 \int_0^T \biggl\langle f(x(t)),
\frac{1}{m} \dot x(t) - f(x(t)) \biggr\rangle \,dt
\biggr),
\]
where $x$ is a deterministic function of $\dot x$ via the relation
$x(t) = x_- + \int_0^t \dot x(s) \,ds$. Since $t\mapsto
f(x(t))$ has bounded variation, we can use partial
integration to get
\begin{eqnarray*}
\int_0^T \langle f(x(t)), d\dot x(t)\rangle
&=& \langle f(x(T)), \dot x(T)\rangle- \langle f
(x(0)), \dot x(0)\rangle\\
&&{}- \int_0^T \langle\dot x(t), Df(x(t)) \dot x(t)\rangle \,dt.
\end{eqnarray*}
Substituting this expression into the formula for $d\tilde\mu_{\dot
x}/d\tilde\nu_{\dot x}$ and using substitution to switch from
$\dot x$ to $x$ gives
\begin{eqnarray*}
\frac{d\tilde\mu_x}{d\tilde\nu_x}(x)
&=& \exp\biggl( m \langle f(x(T)),\dot x(T) \rangle- m
\langle f(x(0)),\dot x(0)\rangle\\
&&\hspace*{20.6pt}{}
- m \int_0^T \langle\dot x(t), Df(x(t)) \dot x(t)\rangle\,
dt\\
&&\hspace*{20.6pt}{}
+ \frac12 \int_0^T \langle f(x(t)), \dot x(t)\rangle-
|f(x(t))|^2 \,dt \biggr),
\end{eqnarray*}
where $\tilde\mu_x = Q_f^{0,x_-}$ is the unconditioned distribution
of $x$ in (\ref{E:target}) and $\tilde\nu_x = Q_0^{0,x_-}$ is the
corresponding distribution for $f=0$. Now we can condition on $x(T)
= x_+$, for example, using \cite{HairerStuartVoss07}, Lemma 5.3, to
get the result.
\end{pf}
\begin{lemma}\label{L:interpolation}
Let $\alpha\in[0,1)$. Then there is a $c>0$ such that
\[
\|\dot x\|_\infty^{1+\alpha}
\leq c \|x\|_\infty^\alpha\|x\|_{\rC^{1+\alpha}}
\]
for all $x\in\rC^1([0,T],\R^d)$.
\end{lemma}
\begin{pf}
The claim for $\alpha=0$ is trivial so we can
assume $\alpha\neq0$.
Assume first the case $d=1$. Write $|\dot x|_\alpha= \sup_{s\neq
t}\frac{|\dot x(t) - \dot x(s)|}{|t-s|^\alpha}$ and let
$t\in[0,T]$ such that $|\dot x(t)| = \|\dot x\|_\infty$. Then
\[
|\dot x|_\alpha
\geq\frac{|\dot x(t) - \dot x(s)|}{|t-s|^\alpha}
\geq\frac{\|\dot x\|_\infty- |\dot x(s)|}{|t-s|^\alpha}
\]
and thus, $|\dot x(s)| \geq\|\dot x\|_\infty- |t-s|^\alpha
|\dot x|_\alpha$ for all $s\in[0,T]$. This allows to conclude that
$|\dot x(t)| \geq\frac12 \|\dot x\|_\infty$ on an interval of
length at least $T \wedge\|\dot x\|_\infty^{1/\alpha} /
(2|\dot x|_\alpha)^{1/\alpha}$. Since we assumed $d=1$, this gives
\[
\|x\|_\infty
\geq\frac12
\cdot\frac12\|\dot x\|_\infty
\cdot\biggl(T \wedge\frac{\|\dot x\|_\infty^{1/\alpha
}}{2^{1/\alpha}|\dot x|_\alpha^{1/\alpha}}\biggr)
= \min\biggl(\frac{T}{4}\|\dot x\|_\infty,
\frac{\|\dot x\|_\infty^{1+1/\alpha}}{2^{2+1/\alpha}|\dot x|_\alpha
^{1/\alpha}}
\biggr)
\]
and by solving this inequality for $\|\dot x\|_\infty$ we find
\[
\|\dot x\|_\infty^{1+\alpha}
\leq c \|x\|_\infty^\alpha
\max( |\dot x|_\alpha, \|x\|_\infty)
\leq c \|x\|_\infty^\alpha\|x\|_{\rC^{1+\alpha}}
\]
for some constant $c$.

For $d>1$ we apply the inequality componentwise: since, for
$z\in\R^d$, we have $\|z\|_2/\sqrt{d} \leq\|z\|_\infty\leq
\|z\|_2$, we get
\begin{eqnarray*}
\|\dot x\|_\infty^{1+\alpha}
&\leq& c \max_{j=1,\ldots, d}
\bigl(\|x_j\|_\infty^\alpha
( \|x_j\|_\infty+ \|\dot x_j\|_\infty+ |x_j|_\alpha)
\bigr)\\
&\leq& c \|x\|_\infty^\alpha
( \|x\|_\infty+ \|\dot x\|_\infty+ |\dot x|_\alpha),
\end{eqnarray*}
where $c$ is increased as needed. This completes the proof.
\end{pf}

The following bound for the density $\varphi$ will be used to show that
the stationary distributions of approximations for the sampling
SPDE (\ref{E:SPDE}) are uniformly integrable.
\begin{lemma}\label{L:U-bound}
Let $\varphi$ be the density from Lemma \ref{L:density}, $U = \log
\varphi$
and $\nu= Q_0^{0,x_-;T,x_+}$ and $\alpha\in(0,1)$. Then for every
$\varepsilon>0$ there is an $M>0$ such that for $\nu$-almost all $x$ we
have
\[
U(x) \leq\varepsilon\|x\|_{\rC^{1+\alpha}}^2 + M.
\]
\end{lemma}
\begin{pf}
We bound the five terms in $U$ one by one. For simplicity we denote
all constants in the following estimates by the symbol $c$, the
meaning of which changes from expression to expression.
\begin{enumerate}[(2)]
\item[(1)] Using the Cauchy--Schwarz inequality we get the bound
\[
\langle f(x_+), \dot x(T) \rangle
\leq|f(x_+)| \|\dot x\|_\infty
\leq|f(x_+)| \|x\|_{\rC^{1+\alpha}}
\leq\varepsilon\|x\|_{\rC^{1+\alpha}}^2 + c.
\]
\item[(2)] A very similar argument gives $-\langle f(x_-),\dot
x(0)\rangle\leq
\varepsilon\|x\|_{\rC^{1+\alpha}}^2 + c$.\vspace*{1pt}
\item[(3)] We can use Young's inequality together with
Lemma \ref{L:interpolation} to conclude that for every $\varepsilon>0$
there is a $c>0$ such that
\[
\|\dot x\|_\infty\leq\varepsilon\|x\|_{\rC^{1+\alpha}} + c \|x\|
_\infty
\]
for all $x\in\rL^2([0,T],\R^d)$. Thus, we have
\begin{eqnarray*}
&&
-\int_0^T \langle Df(x(t)) \dot x(t), \dot x(t) \rangle \,dt\\
&&\qquad\leq \biggl\| \frac{d}{dt} f(x) \biggr\|_2
\| \dot x \|_2
\leq T \biggl\| \frac{d}{dt} f(x) \biggr\|_\infty
\| \dot x \|_\infty\\
&&\qquad\leq \bigl( \varepsilon\|f(x)\|_{\rC^{1+\alpha}} + c \|f(x)\|
_\infty\bigr)
( \varepsilon\|x\|_{\rC^{1+\alpha}} + c \|x\|_\infty).
\end{eqnarray*}
Since $f$ is differentiable with bounded derivatives, we have
$\|f(x)\|_{\rC^{1+\alpha}} \leq c \|x\|_{\rC^{1+\alpha}} + c$ and
by assumption
there is a $\beta<1$ such that $|f(x)| \leq|x|^\beta+c$. Using
these estimates we find
\[
-\int_0^T \langle Df(x(t)) \dot x(t), \dot x(t) \rangle \,dt
\leq( \varepsilon\|x\|_{\rC^{1+\alpha}} + c \|x\|_\infty
^\beta+ c )
( \varepsilon\|x\|_{\rC^{1+\alpha}} + c \|x\|_\infty)
\]
and thus, for every $\varepsilon>0$ there is a $c>0$ such that
\[
-\int_0^T \langle Df(x(t)) \dot x(t), \dot x(t) \rangle \,dt
\leq\varepsilon\|x\|_{\rC^{1+\alpha}}^2 + \varepsilon\|x\|_\infty
^2 + c \|x\|_\infty^{1+\beta} + c.
\]
Since $\beta<1$, this gives the required bound.
\item[(4)] Using the Cauchy--Schwarz inequality again, we get in a
similar way
\[
\int_0^T \langle f(x(t)), \dot x(t)\rangle \,dt
\leq\|f(x)\|_2 \|\dot x\|_2
\leq c ( \|x\|_\infty^\beta+ c ) \|\dot x\|_\infty
\leq\varepsilon\|x\|_{\rC^{1+\alpha}}^2 + c.
\]
\item[(5)] Finally, we have $-\int_0^T |f(x(t))|^2 \,dt
< 0$.
\end{enumerate}
Combining these bounds gives the required result.
\end{pf}
\begin{lemma}\label{L:N-regularity}
The drift $\CN$ defined by (\ref{E:drift}) is locally Lipschitz from
$\CH_{1/4}$ to $\CH_{-7/16}$. Furthermore, one can write $\CN=
\CN_1 + \CN_2 + \CN_3$ such that $\CN_1$ does not depend on $x$ and
such that the bounds
% \minilab{e:nonlinear}
%
%e15 ###
%
\begin{subequation}
\begin{eqnarray}
\label{e:nonlineartotal}
\|\CN(x)\|_{\CH_{-7/16}} &\le& c (1 + \|x\|_{\CH_{1/4}}^2
),\\
\label{e:nonlinear2}
\|\CN_2(x) - \CN_2(y)\|_{\CH_{-5/16}} &\le& c \|x-y\|_{\CH_{1/4}}
(\|x\|_{\CH_{1/4}}+ \|y\|_{\CH_{1/4}}),\\
\label{e:nonlinear3}
\|\CN_3(x) - \CN_3(y)\|_{\CH_{-3/16}} &\le& c \|x-y\|_{\CH_{1/4}}
(\|x\|_{\CH_{1/4}}+ \|y\|_{\CH_{1/4}})^2
\end{eqnarray}
\end{subequation}
hold for all pairs $x,y\in\CH_{1/4}$ and for some constant $c>0$.
\end{lemma}
\begin{pf}
We use the characterization of the spaces $\CH_\alpha$ from
Corollary \ref{cor:inclinterp} and, in particular, the fact that if $x
\in\CH_{1/4}$, then $x$ also belongs to $\rH^1$. Since we assumed
that $f_j$ and its derivatives up to the second order are globally
Lipschitz, this implies that $f_j(x)$, $\partial_i f_j(x)$ and
$\partial_{ij}
f_k(x)$ all belong to $\rH^1$ and their norms are bounded by multiples
of that of $x$.

Now let $x\in\rH^1$. Then the following statements hold:
\begin{itemize}
\item$f_i(x) \,\partial_k f_i(x) \in\rH^1$ since $\rH^1$ is stable under
composition with smooth functions,
\item$\partial_t
x_i [\partial_i f_k(x) - \partial_k f_i(x)] \in\rL^2$,
for the same reason,
\item$\partial_t x_i \,\partial_t x_j \,\partial_{ij}^2 f_k(x)
\in\rH^t$ for all
$t<-1/2$ since, in this case, $\rL^1 \subseteq\rH^t$ by Sobolev embedding,
\item$\partial_t^2 x_j [\partial_j f_k(x) + \partial_k f_j(x)
]\in\rH^{-1}$ since
$\rH^{-1}$ is stable under multiplication by $\rH^1$-functions,
\item$f_k(x_-)\partial_t\delta_0 \in\rH^t$ and
$f_k(x_+)\partial_t\delta_T \in\rH^t$ for every $t<-3/2$.
\end{itemize}
It follows that $\CN$ maps $\CH_{1/4}$ into $\CH_{\alpha}$ for every
$\alpha< -{3\over8}$. In particular, it maps $\CH_{1/4}$ into $\CH_{-{7
/16}}$ as stated and the bound (\ref{e:nonlineartotal}) holds.
We then define $\CN_1$ as the term proportional to $f_k(x_-)\,\partial
_t\delta_0
- f_k(x_+)\,\partial_t\delta_T$, $\CN_3$ as the term proportional to
$\partial_t x_i\,
\partial_t x_j \,\partial_{ij}^2 f_k(x)$ and $\CN_2$ as the sum of
the remaining
terms in the nonlinearity. With these definitions at hand, the bounds
(\ref{e:nonlinear2}) and (\ref{e:nonlinear3}) follow easily.
\end{pf}
\begin{proposition}\label{P:local}
For every initial condition $x_0\in\rL^2([0,T], \R^d)$,
the stochastic evolution equation (\ref{E:SPDE}) has a unique
maximal local solution $(x,\tau^*)$. The solution satisfies
$x(\tau) \in\CH_{1/4}$ for every $\tau<\tau^*$ a.s. and
$\sup_{\tau\uparrow\tau^*} \|x(\tau)\|_{\rL^2} = \infty$ a.s. on
the set $\{\tau^*<\infty\}$.
\end{proposition}
\begin{pf}
Define
\[
g(\tau) = S(\tau) x_0 + \sqrt{2} \int_0^\tau S(\tau-\sigma)\,
dw(\sigma).
\]
Let $R>U>0$. For $x\dvtx(0,U]\to\CH_{1/4}$ continuous, define
\[
\| x \|_* = \sup_{\tau\in(0,U]} \tau^{1/4} \|x(\tau)\|_{\CH_{1/4}}
\]
and let $\CX$ be the space of all such $x$ with $x(0)=g(0)$ and $\|x\|_*
< \infty$. Then $(\CX$, $\mbox{$\|\quark\|_*$})$ is a Banach space.
We find
\begin{eqnarray*}
\|g\|_*
&\leq&
\sup_{\tau\in(0,U]} \tau^{1/4} \biggl( \frac{1}{\tau^{1/4}} \|
x_0\|_{\rL^2} + \sqrt{2}\biggl \| \int_0^\tau S(\tau-\sigma)\,
dw(\sigma) \biggr\|_{\CH_{1/4}} \biggr) \\
&\leq&\|x_0\|_{\rL^2} + \sqrt{2} R \sup_{\tau\in[0,R]} \biggl\|
\int_0^\tau S(\tau-\sigma) \,dw(\sigma) \biggr\|_{\CH_{1/4}} \\
&=&\!: \|x_0\|_{\rL^2} + C_R
\end{eqnarray*}
and thus, $g\in\CX$ for every $U<R$. Define a map $M_g\dvtx\CX\to
\CX$ by
\[
M_g x(\tau)
= \int_0^\tau S(\tau-\sigma) \CN(x_\sigma) \,d\sigma+
g(\tau)\qquad
\forall\tau\in[0,U].
\]
By the definition of a mild solution, local solutions up to time $U$
coincide with the fixed points of this map.

Let $B(g,1) \subseteq\CX$ denote the closed ball around $g$ with
radius 1. By Lemma \ref{L:N-regularity}, the nonlinearity
$\CN\dvtx\CH_{1/4}\to\CH_{-7/16}$ is locally Lipschitz and thus,
for all $x,y\in B(g,1)$, we have
\begin{eqnarray*}
&&\| M_gx - M_gy \|_*\\
&&\qquad\leq\sup_{\tau\in(0,U]} \tau^{1/4}
\int_0^\tau\bigl\| S(\tau-\sigma) \bigl(\CN(x_\sigma) - \CN
(y_\sigma)\bigr) \bigr\|_{\CH_{1/4}} \,d\sigma\\
&&\qquad\leq\sup_{\tau\in(0,U]} c \tau^{1/4}
\int_0^\tau\biggl(\frac{\| \CN_2(x_\sigma) - \CN_2(y_\sigma
)\|_{\CH_{-5/16}}}{(\tau-\sigma)^{9/16}}\\
&&\qquad\quad\hspace*{76.4pt}{} + \frac{\| \CN
_3(x_\sigma) - \CN_3(y_\sigma)\|_{\CH_{-3/16}}}{(\tau-\sigma
)^{7/16}} \biggr) \,d\sigma\\
&&\qquad\leq\sup_{\tau\in(0,U]} c\tau^{1/4}
\int_0^\tau\| x_\sigma- y_\sigma\|_{\CH_{1/4}}
\biggl(\frac{\| x_\sigma\|_{\CH_{1/4}} + \|y_\sigma\|_{\CH_{1/4}}}{(\tau
-\sigma)^{9/16}}\\
&&\qquad\quad\hspace*{143pt}{}
+ \frac{ (\| x_\sigma\|_{\CH_{1/4}} + \|y_\sigma\|_{\CH
_{1/4}})^2}{(\tau-\sigma)^{7/16}}\biggr) \,d\sigma\\
&&\qquad\leq c U^{1/16} \|x - y\|_* (1 + \|x\|_* + \|y\|_*)^2,
\end{eqnarray*}
where $c$ changes from line to line.
Similarly, we have
\[
\| M_gx - g\|_* \le c U^{1/16} \|x\|_*^2 \leq c U^{1/16}
( \|x-g\|_* + \|g\|_*)^2.
\]
By choosing the final time $U$ sufficiently small, we can then make sure
that $M_g$ is a contraction on the ball $B(g,1)$ and, by the Banach fixed
point theorem, $M_g$ has a unique fixed point. This gives a unique local
solution of (\ref{E:SPDE}) up to time $U$.

By iterating this procedure, every time starting with the final
point of the previously constructed segment, we obtain a solution up
to a maximal time $\tau^*\leq R$. Since the length of each segment of
this solution only depends on the $\rL^2$-norm of its starting point,
we see that $\tau^*<R$ implies $\sup_{\tau<\tau^*} \|x(\tau)\|
_{\rL^2} =
\infty$. Taking $R\to\infty$ completes the proof.
\end{pf}

Even if $f$ is globally Lipschitz, the $\partial_t x_i \,\partial_t
x_j\,\partial_{ij}^2 f_k$-term causes the nonlinearity $\CN$ to be only locally
Lipschitz. Thus, showing the existence of global solutions to the
SDE (\ref{E:SPDE}) will need some care.
\begin{proposition}\label{P:global}
For every initial condition $x_0\in\rL^2([0,T], \R^d)$ the
SPDE (\ref{E:SPDE}) has a unique global solution. For every
$\tau>0$ the solution satisfies\break $\E(\|x(\tau)\|_{\rL
^2}^2) <
\infty$.
\end{proposition}
\begin{pf}
From Proposition \ref{P:local} we know that (\ref{E:SPDE}) has a
local solution $(x,\tau_{\max})$. Let $y$ be the solution of the
linear SPDE from Proposition \ref{P:linear}, that is,
\[
y(\tau) = S(\tau)(x_0 - \bar x)
+ \sqrt{2} \int_0^\tau S(\tau- \sigma) \,dw(\sigma) + \bar x
\]
and define $z(\tau) = x(\tau) - y(\tau)$ for every $\tau\in[0,\tau
_{\max})$.
Then $z$ satisfies the stochastic evolution equation
\[
dz(\tau) = \CL z(\tau) \,d\tau+ \CN\bigl(z(\tau) + y(\tau)
\bigr) \,d\tau,\qquad
z(0) = 0.
\]
Thus $\|z(\tau)\|_{\rL^2}^2$ satisfies
%
%e16 ###
%
\begin{eqnarray}\label{E:z-norm-ODE}
\frac{d \|z(\tau)\|_{\rL^2}^2}{d\tau}
&=& 2 \bigl\langle z(\tau), \CL z(\tau)
+ \CN\bigl(z(\tau)+y(\tau)\bigr) \bigr\rangle\nonumber\\
&=& -4m^2 \langle\partial_t^2 z(\tau), \partial_t^2 z(\tau) \rangle
- 4m | \partial_t z(0) |^2
- 4m | \partial_t z(1) |^2 \nonumber\\[-8pt]\\[-8pt]
&&{}
- \langle\partial_t z(\tau), \partial_t z(\tau) \rangle
+ 2 \bigl\langle z(\tau), \CN\bigl(z(\tau)+y(\tau)\bigr)
\bigr\rangle
\nonumber\\
&\leq&- c \| z(\tau) \|_{\rH^2}^2
+ 2 \bigl\langle z(\tau), \CN\bigl(z(\tau)+y(\tau)\bigr)
\bigr\rangle
\nonumber
\end{eqnarray}
for some $c>0$. This formal calculation can be made rigorous by a
standard approximation argument, using, for example, Galerkin
approximations.

We require {a priori} bounds of the form $\langle z, \CN(z+y)
\rangle\leq c \|z\|_{\rL^2}^2 + \varepsilon\|z\|_{\rH^2}^2 +
c$ where $\varepsilon>0$ is
small enough to be compensated by the negative $\|z\|_{\rH^2}^2$-term
in~(\ref{E:z-norm-ODE}).\vspace*{2pt}

In order to obtain the required bounds, we consider the five terms
from the definition of $\CN$ individually. For the purpose of these
estimates we denote all numerical constants by $c>0$ and only track
the $y$-dependency of the bounds explicitly. For the first term we
get
\[
\langle z_k, -f_i(z+y) \,\partial_k f_i(z+y)\rangle
\leq c \|z\|_{\rL^2} \|f(z+y)\|_{\rL^2}
\leq c \|z\|_{\rL^2}^2 + c \|y\|_{\rL^2}^2 + c.
\]
For the second term we find
\begin{eqnarray*}
&&
\langle z_k,
\partial_t(z_i+y_i) \,\partial_t(z_j+y_j) \,\partial^2_{ij} f_k(z+y)
\rangle\\
&&\qquad= \int_0^1 z_k \,\partial_tz_i \,\partial_t\bigl(\partial
_if_k(z+y)\bigr) \,dt
+ \int_0^1 z_k \,\partial_t y_i \,\partial_t(z_j+y_j) \,\partial
^2_{ij} f_k(z+y) \,dt \\
&&\qquad= - \int_0^1 \partial_t z_k\, \partial_tz_i \,\partial_if_k(z+y) \,dt
- \int_0^1 z_k \,\partial_t^2 z_i \,\partial_if_k(z+y) \,dt \\
&&\qquad\quad{}
+ \int_0^1 z_k \,\partial_t y_i \,\partial_t(z_j+y_j) \,\partial
^2_{ij} f_k(z+y) \,dt \\
&&\qquad\leq c \|z\|_{\rH^1}^2 + c \|z\|_{\rL^2} \|z\|_{\rH^2} + c \|z\|
_{\rL^\infty} \|y\|_{\rH^1}( \|z\|_{\rH^1} + \|y\|_{\rH
^1}).
\end{eqnarray*}
For the third term we have
\[
\langle z_k, -\partial_t(z_i+y_i) (\partial_i f_k - \partial
_k f_i) \rangle
\leq c \|z\|_{\rL^2} ( \|z\|_{\rH^1} + \|y\|_{\rH^1} ).
\]
The fourth term can be bounded as
\begin{eqnarray*}
&&
\langle z_k,
\partial_t^2(z_i+y_i) (\partial_i f_k + \partial_k f_i)
\rangle\\
&&\qquad= \langle z_k,
\partial_t^2 z_i (\partial_i f_k + \partial_k f_i) \rangle\\
&&\qquad\quad{}
- \int_0^1 \partial_t z_k \,\partial_t y_i (\partial_i f_k +
\partial_k f_i) \,dt\\
&&\qquad\quad{}
- \int_0^1 z_k \,\partial_t y_i \,\partial_t(z_j+y_j) (\partial
^2_{ij} f_k + \partial^2_{jk} f_i) \,dt \\
&&\qquad\leq c \|z\|_{\rL^2} \|z\|_{\rH^2} + c \|z\|_{\rH^1} \|y\|_{\rH^1}
+ c \|z\|_{\rL^\infty} \|y\|_{\rH^1}( \|z\|_{\rH^1}+ \|y\|
_{\rH^1}).
\end{eqnarray*}
Finally, for the fifth term involving the derivatives of Dirac
distributions, we get
\[
\langle z, \partial_t \delta_0\rangle
= - z'(0)
\leq c \|z\|_{\rH^\alpha},\qquad
\langle z, -\partial_t \delta_1\rangle
= z'(1)
\leq c \|z\|_{\rH^\alpha}
\]
for every $\alpha>3/2$.

To convert the bounds into the required form first note that for
every $s\in(0,2)$ the interpolation inequality (see, e.g.,
\cite{Hairer08}, Corollary 6.11) gives $\|z\|_{\rH^s}^2 \leq
\|z\|_{\rL^2}^{2-s} \|z\|_{\rH^2}^s$ and, using Young's inequality, we
can, for every $\varepsilon>0$, find a $c>0$ such that
\[
\|z\|_{\rL^2}^{2-s} \|z\|_{\rH^2}^s
\leq c \|z\|_{\rL^2}^2 + \varepsilon\|z\|_{\rH^2}^2.
\]
Using this relation we find a $c>0$ such that
\begin{eqnarray*}
\|z\|_{\rL^\infty} \|y\|_{\rH^1}^2
&\leq&\tfrac12 \|z\|_{\rL^\infty}^2 + \tfrac12 \|y\|_{\rH^1}^4
\leq c \|z\|_{\rH^1}^2 + c \|y\|_{\rH^1}^4\\
&\leq& c \|z\|_{\rL^2}^2 + \varepsilon\|z\|_{\rH^2}^2 + c \|y\|_{\rH^1}^4.
\end{eqnarray*}
The terms of the form $\|z\|_{\rL^\infty} \|z\|_{\rH^1} \|y\|_{\rH
^1}$ can
be bounded using the relation
\[
\|z\|_{\rL^\infty} \|z\|_{\rH^1}
\leq c \|z\|_{\rH^{3/4}} \|z\|_{\rH^1}
\leq c \|z\|_{\rL^2}^{5/8} \|z\|_{\rH^2}^{3/8} \|z\|_{\rL
^2}^{1/2} \|z\|_{\rH^2}^{1/2}
= c \|z\|_{\rL^2}^{9/8} \|z\|_{\rH^2}^{7/8}.
\]
Applying Young's inequality with $p=16/7$ and $q=16/9$ we find a
$c>0$ such that
\[
\|z\|_{\rL^\infty} \|z\|_{\rH^1} \|y\|_{\rH^1}
\leq c \|z\|_{\rH^2}^{7/8} \|z\|_{\rL^2}^{9/8} \|y\|_{\rH^1}
\leq\varepsilon\|z\|_{\rH^2}^2 + c \|z\|_{\rL^2}^2 \|y\|_{\rH^1}^{16/9}.
\]

Combining all these estimates, we find that for every $\varepsilon>0$ there
is a $c>0$ such that
\[
\bigl\langle z(\tau), \CN\bigl(z(\tau)+y(\tau)\bigr)
\bigr\rangle
\leq c (1 + \|y\|_{\rH^1}^{16/9}) \|z\|_{\rL^2}^2
+ \varepsilon\|z\|_{\rH^2}^2
+ c (1 + \|y\|_{\rL^2}^2 + \|y\|_{\rH^1}^4)
\]
and substituting this bound into (\ref{E:z-norm-ODE}) for small
enough $\varepsilon>0$ we get
\[
\frac{d \|z(\tau)\|_{\rL^2}^2}{d\tau}
\leq c (1 + \|y\|_{\rH^1}^{16/9})\|z(\tau)\|
_{\rL^2}^2
+ c (1 + \|y\|_{\rL^2}^2 + \|y\|_{\rH^1}^4).
\]
Gronwall's inequality gives
%
%e17 ###
%
\begin{eqnarray}\label{E:Gronwall}
\|z(\tau)\|_{\rL^2}^2
&\leq& c \int_0^\tau\bigl(1 + \|y(\sigma)\|_{\rH^1}^{16/9}\bigr)
\int_0^\sigma\bigl(1 + \|y(r)\|_{\rL^2}^2 + \|y(r)\|_{\rH
^1}^4\bigr) \,dr \nonumber\\
&&\hspace*{20pt}{}\times
\exp\biggl(\int_\sigma^\tau\bigl(1 + \|y(r)\|_{\rH^1}^{16/9}
\bigr) \,dr\biggr) \,d\sigma\\
&&{}
+ c \int_0^\tau\bigl(1 + \|y(\sigma)\|_{\rL^2}^2 + \|y(\sigma)\|
_{\rH^1}^4\bigr) \,d\sigma.
\nonumber
\end{eqnarray}
Thus, $\|z\|_{\rL^2}$ cannot explode in finite time and from
Proposition \ref{P:local} we get $\tau_{\max}=\infty$.

By Proposition \ref{P:linear} we have $y\in
\rL^2([0,\tau],\rH^1)$. Hence, by Fernique's theorem (see,
e.g., \cite{Hairer08}, Theorem 3.11),
\[
\E\biggl(\exp\biggl(\varepsilon\int_0^\tau\|y(r)\|_{\rH^1}^2
\,dr\biggr) \biggr)
< \infty
\]
for sufficiently small $\varepsilon>0$. Thus, using the fact that
$16/9 <
2$, we see that the right-hand side of (\ref{E:Gronwall}) has finite
expectation for all $\tau>0$.
\end{pf}

Now the only part of Theorem \ref{T:answer} which we still need to
prove is the statement about the stationary distribution
of (\ref{E:SPDE}). This can be done using a finite-dimensional
approximation argument, similar to the proofs in \cite{Zabczyk88}
and \cite{HairerStuartVoss07}, Section 3. Since these articles
assumed that $U$ was bounded from above and also assumed different
regularity properties for the drift, the proof needs to be adapted for
the situation here; to allow for easier reading, we include the full
argument instead of just enumerating the required changes.
\begin{proposition}\label{P:stationary}
The distribution $Q_f^{0,x_-;T,x_+}$ is invariant for (\ref{E:SPDE}).
\end{proposition}
\begin{pf}
Let $\varphi$ be the density of $\mu=Q_f^{0,x_-;T,x_+}$ w.r.t.
$\nu=Q_0^{0,x_-;T,x_+}$ as given by Lemma \ref{L:density} and let $U
= \log\varphi$. Then we can compute the derivative of $U$ at
$x\in\CH_{1/4}$ in direction $h\in\CH_{7/16}$ as
\begin{eqnarray*}
\langle DU(x), h \rangle
&=& m f_k(x_+) \dot h_k(T) - m f_k(x_-) \dot h_k(0) \\
&&{} + 2 \int_0^T \biggl(
- f_i \,\partial_k f_i
+ m \dot x_i \dot x_j \,\partial_{ij}^2 f_k
- \frac12 \dot x_i (\partial_i f_k - \partial_k f_i)\\
&&\hspace*{147.6pt}{}+ m\ddot x_i (\partial_i f_k + \partial_k f_i)
\biggr) h_k(t) \,dt \\
&=& \langle\CN(x), h \rangle.
\end{eqnarray*}
Here we used the fact that, by Corollary \ref{cor:inclinterp}, $h\in
\CH_{7/16}$ implies $h(0) = h(T) = 0$. This shows that the
function $U$ is Fr\'echet-differentiable with derivative $\CN$. Let
$\Pi_n$ and $\hat\Pi_n$ be as in Lemma \ref{L:pi-n-hat} and define
the approximations
\[
\CN_n = (U\circ\hat\Pi_n)' = \hat\Pi_n \CN(\hat\Pi
_n \quark)
\]
for $n\in\N$.

Consider the $n$-dimensional SDEs
\[
dy_n(\tau) = \CL y_n(\tau) \,d\tau + \sqrt{2} \Pi_n \,dw(\tau),\qquad
y_n(0) = \Pi_n x_0,
\]
and
\[
dx_n(\tau) = \CL x_n(\tau) \,d\tau+ \CN_n(x_n(\tau))\,
d\tau+ \sqrt{2} \Pi_n \,dw(\tau),\qquad
x_n(0) = \Pi_n x_0.
\]
Then, by finite-dimensional results, the stationary distributions
$\nu_n$ and $\mu_n$ of $y_n$ and $x_n$, respectively, are given by
\[
\nu_n = \nu\circ\Pi_n^{-1}
\quad\mbox{and}\quad
\frac{d\mu_n}{d\nu_n} = \exp(U\circ\hat\Pi_n).
\]
Define the semigroup $(\CP^n_\tau)_{\tau\geq0}$ on
$C_{\mathrm{b}}(\CH, \R)$ by $\CP^n_\tau\varphi(x) = \E_x
(\varphi(x_n(\tau)))$ for all $x\in E_n$ and $\varphi\in
C_{\mathrm{b}}(\CH, \R)$. Since the process $x_n$ is
$\mu_n$-reversible, we have
%
%e18 ###
%
\begin{equation}\label{E:reversible-n}
\int_{\CH} \varphi(x) \CP^n_\tau\psi(x) \,d\mu_n(x)
= \int_{\CH} \psi(x) \CP^n_\tau\varphi(x) \,d\mu_n(x)
\end{equation}
for every $\varphi, \psi\in C_{\mathrm{b}}(\CH, \R)$.

We need to find the limit of (\ref{E:reversible-n}) as $n\to\infty$.
For this, we first show that $x_n \to x$ in $\CH_{1/4}$ uniformly
on bounded time intervals. Let $U>0$, then we have
\begin{eqnarray*}
\|x_n(\tau)-x(\tau)\|_{\CH_{1/4}}
&\leq&\biggl\| ( \Pi_n - I ) \biggl( S(\tau) x_0
+ \sqrt{2} \int_0^\tau S(\tau-\sigma) \,dW(\sigma)
\biggr) \biggr\|_{\CH_{1/4}} \\
&&{}
+ \biggl\| \int_0^\tau S(\tau-\sigma)
\bigl( \CN_n(x(\sigma)) - \CN(x(\sigma))\bigr) \,d\sigma\biggr\|
_{\CH_{1/4}} \\
&&{}
+ \biggl\| \int_0^\tau S(\tau-\sigma)
\bigl( \CN_n(x_n(\sigma)) - \CN_n(x(\sigma))\bigr) \,d\sigma
\biggr\|_{\CH_{1/4}} \\
&=&\!: I_1(\tau) + I_2(\tau) + I_3(\tau)
\end{eqnarray*}
for all $\tau\in[0,U]$.

From the definition of $\|\quark\|_{\CH_\alpha}$ and the asymptotics
of the eigenvalues of $\CL$ in Lemma \ref{L:eigenvalues} we get, for
any $\beta>\alpha$, that there is a $c>0$ such that the bound
\[
\| \Pi_n x - x \|_{\CH_\alpha}
\leq\frac{c}{n^{8(\beta-\alpha)}} \| x \|_{\CH_\beta}
\]
holds for all $x\in\CH_\beta$ and all $n\in\N$. Let
$\beta\in(1/4,3/8)$. Then we know from Proposition \ref{P:linear}
that $\tau\mapsto S(\tau) x_0 + \sqrt{2} \int_0^\tau S(\tau-\sigma)\,
dW(\sigma)$ is a continuous map from $[0,U]$ into $\CH_\beta$.
Combining these two statements, we find $\sup_{0\leq\tau\leq U}
I_1(\tau) \to0$ as $n\to\infty$.

From Lemma \ref{L:N-regularity} we know that $\CN$ is locally
Lipschitz from $\CH_{1/4}$ to $\CH_{-7/16}$. By
Lemma \ref{L:pi-n-hat}, part (c), there is then a constant $K_r>0$
such that
\[
\| \CN_n(x) - \CN_n(y) \|_{\CH_{-7/16}}
\leq K_r \| x - y \|_{\CH_{1/4}}
\]
for all $n\in\N$ and all $x$ and $y$ with
$\|x\|_{\CH_{1/4}},\|y\|_{\CH_{1/4}} \leq r$. Thus, the $\CN_n$ are
also locally Lipschitz.

We can find $p,q>1$ such that $p \cdot\frac{11}{16} < 1$ and $1/p +
1/q = 1$. For $I_2$ we then get
\begin{eqnarray*}
I_2(\tau)
&\leq&\int_0^\tau\bigl\| S(\tau-\sigma)
\bigl( \CN_n(x(\sigma)) - \CN(x(\sigma))\bigr) \bigr\|_{\CH
_{1/4}} \,d\sigma\\
&\leq&\int_0^\tau\| S(\tau-\sigma) \|_{\CH_{-7/16}\to
\CH_{1/4}}
\| \CN_n(x(\sigma)) - \CN(x(\sigma)) \|_{\CH_{-7/16}}\,
\,d\sigma\\
&\leq& c \biggl( \int_0^U \frac{1}{\sigma^{p 11/16}} \,d\sigma
\biggr)^{1/p}
\biggl( \int_0^U \| \CN_n(x(\sigma)) - \CN(x(\sigma))
\|_{\CH_{-7/16}}^q\,
d\sigma\biggr)^{1/q}.
\end{eqnarray*}
The right-hand side is independent of $\tau$ and converges to $0$ as
$n\to\infty$ by dominated convergence, using Lemma \ref{L:pi-n-hat},
part (b).

For $n\in\N$ define
\[
T_{n,r} = \inf\{\tau\in[0,U]
\mid \|x(\tau)\| > r \mbox{ or } \|x_n(\tau)\| > r
\}
\]
with the convention $\inf\varnothing= U$. For $\tau\leq T_{n,r}$ we
have
\[
I_3(\tau)
\leq K_r \int_0^\tau\| S(\tau-\sigma) \|_{\CH
_{-7/16}\to\CH_{1/4}} \| x_n(\sigma) - x(\sigma) \|
_{\CH_{1/4}} \,d\sigma
\]
and consequently
\begin{eqnarray*}
\|x_n(\tau)-x(\tau)\|_{\CH_{1/4}}
&\leq&\sup_{0\leq\sigma\leq U} \bigl( I_1(\sigma) + I_2(\sigma)
\bigr) \\
&&{} + c K_r \int_0^\tau\frac{1}{(\tau-\sigma)^{11/16}}
\| x_n(\sigma) - x(\sigma) \| \,d\sigma.
\end{eqnarray*}
Using Gronwall's lemma we can conclude
\[
\|x_n(\tau)-x(\tau)\|_{\CH_{1/4}}
\leq\sup_{0\leq\sigma\leq U} \bigl( I_1(\sigma) + I_2(\sigma)
\bigr)
\cdot
\exp\biggl(c K_r \int_0^U \frac{1}{\sigma^{11/16}} \,d\sigma\biggr)
\]
for all $\tau\leq T_{n,r}$. As we have already seen, the right-hand
side converges to $0$ as $n\to\infty$.

Now choose $r>0$ big enough such that ${\sup_{0\leq
\tau\leq U} }\|x(\tau)\| \leq r/4$. Then for sufficiently
large $n$ and all $\tau\leq T_{n,r}$ we have $\|x_n(\tau) -
x(\tau)\| \leq r/4$ and thus, ${\sup_{0\leq\tau\leq T_{n,r}}}
\|x_n(\tau)\| \leq r/2$. This implies $T_{n,r} = U$ for
sufficiently large $n$. Thus, we have $x_n \to x$ in
$\rC([0,U],\CH_{1/4})$ a.s.

Let $0 < \alpha< \beta< 1/2$. Define the semigroup
$(\CP_\tau)_{\tau\geq0}$ on $C_{\mathrm{b}}(\CH, \R)$ by
$\CP_\tau\varphi(x) = \E_x (\varphi(x(\tau)))$ for
all $x\in
\CH$ and $\varphi\in C_{\mathrm{b}}(\CH, \R)$. Then, by dominated
convergence, we have $\CP^n_\tau\varphi(\Pi_n x) \to\CP_\tau
\varphi(x)$
as $n\to\infty$. By Lemma \ref{L:stat-is-Hoelder}, $x\in
\rC^{1+\beta}$ for $\nu$-almost all $x$. Furthermore, $U\dvtx
\rC^{1+\alpha} \to\R$ is continuous and thus $U(\hat\Pi_n x) \to
U(x)$ as $n\to\infty$ for $\nu$-almost all $x$ by
Lemma \ref{L:pi-n-hat}, part (e).

Finally, let $c = \| \hat\Pi_n \|_{\rC^{1+\beta}_0 \to
\rC^{1+\alpha}_0}$. Using Fernique's theorem we can choose
$\varepsilon>0$
such that the function $\exp(\varepsilon c
\|x\|_{\rC^{1+\beta}}^2)$ is $\nu$-integrable. By
Lemma \ref{L:U-bound} we can find an $M>0$ such that $U(\hat\Pi_n x)
\leq\varepsilon\|\hat\Pi_n x\|_{\rC^{1+\alpha}}^2 + M \leq
\varepsilon c
\|x\|_{\rC^{1+\beta}}^2 + M$ for all $n\in\N$ and $\nu$-almost
all $x$. Then dominated convergence gives
\begin{eqnarray*}
\lim_{n\to\infty}
\int_{\CH} \varphi(x) \CP^n_\tau\psi(x) \,d\mu_n(x)
&=& \lim_{n\to\infty}
\int_{\CH} \varphi(\Pi_n x) \CP^n_\tau\psi(\Pi_nx)
\rmE^{U(\hat\Pi_n x)} \,d\nu(x) \\
&=& \int_{\CH} \varphi(x) \CP_\tau\psi(x) \rmE^{U(x)}
\,d\nu(x)\\
&=& \int_{\CH} \varphi(x) \CP_\tau\psi(x) \,d\mu(x)
\end{eqnarray*}
and using (\ref{E:reversible-n}) we get
\[
\int_{\CH} \varphi(x) \CP_\tau\psi(x) \,d\mu(x)
= \int_{\CH} \psi(x) \CP_\tau\varphi(x) \,d\mu(x).
\]
Thus, the process $x$ is $\mu$-reversible which is the required
result.
\end{pf}

Propositions \ref{P:local}, \ref{P:global} and \ref{P:stationary}
together imply all claims of Theorem \ref{T:answer} and so the proof
of the result is complete.

\section*{Acknowledgment}
The third named author thanks the Courant Institute, where
part of this article was completed.

% imsref loaded by lrinkeviciute, 2010-11-03 08:25:05
%
% imsref loaded by lrinkeviciute, 2010-11-03 11:00:36

%
\printaddresses

\end{document}